\documentclass[reqno]{amsart}

\usepackage{style}

\author{Gregor Kemper}
\address{Technische Universit\"at M\"unchen, Zentrum Mathematik - M11,
Boltzmannstr. 3, 85748 Garching, Germany}
\email{kemper@ma.tum.de}

\title{Quotients by Connected Solvable Groups}


\begin{document}

\begin{abstract}
  This paper introduces the notion of an excellent quotient, which is
  stronger than a universal geometric quotient. The main result is
  that for an action of a connected solvable group $G$ on an affine
  scheme $\Spec(R)$ there exists a semi-invariant~$f$ such that
  $\Spec(R_f) \to \Spec\bigl((R_f)^G\bigr)$ is an excellent
  quotient. The paper contains an algorithm for computing~$f$ and
  $(R_f)^G$. If $R$ is a polynomial ring over a field, the algorithm
  requires no Gr\"obner basis computations, and it also computes a
  presentation of $(R_f)^G$. In this case, $(R_f)^G$ is a complete
  intersection. The existence of an excellent quotient extends to
  actions on quasi-affine schemes.
\end{abstract}

\maketitle

\section*{Introduction} \label{sIntro}%

In the theory of connected algebraic groups, two cases stand out as
being well understood: reductive groups and solvable groups. While the
invariant theory of reductive groups is well-behaved and, in many
aspects, well understood, this is not the case for solvable and, in
particular, unipotent groups. For example, invariant rings of
unipotent groups need not be finitely generated, and even even if they
are, categorical quotients need not exist (see
\mycite[Example~4.10]{Ferrer:Santos:Rittatore})). However, a result of
\mycite{Rosenlicht:63} tells us that any variety $X$ with an action of
an algebraic group has a dense open subset $U \subseteq X$ that admits
a geometric quotient. A constructive version, involving huge Gr\"obner
basis computations, can by found in
\mycite{Kemper.Rosenlicht.07}. This raises the question if more can be
said for actions of special classes of groups, and if computations
become easier for such groups. This brings us back to the case of
unipotent groups, for which some further reaching results have indeed
been obtained. In fact, quite a few authors have studied invariant
theory of unipotent groups, e.g. \mycite{Hochschild:Mostow:73},
\mycite{Grosshans:83},
\citename{Fauntleroy:85}~[\citenumber{Fauntleroy:85},
\citenumber{Fauntleroy:88}], and \mycite{Berczi:Doran:Hawes:Kirwan};
but the papers on the subject that are relevant in our context are
\citename{Greuel:Pfister:1993}~[\citenumber{Greuel:Pfister:1993},
\citenumber{Greuel:Pfister:1998}] and \mycite{Sancho:2000}. Among
other results, these papers contain the following key statement: If a
connected unipotent group acts on $X$, there is a nonzero
invariant~$f$ such that $X_f$ admits a geometric quotient $X_f \to Y$.
(More specifically, in [\citenumber{Greuel:Pfister:1993}] $X$ is a
quasi-affine scheme over a field of characteristic~$0$ and the authors
also show that $X_f \cong \AA^n \times Y$ as schemes over $Y$, while
in ~[\citenumber{Sancho:2000}] $X$ is an affine scheme over a field of
any characteristic; but see \rref{2rGPS} below about these
statements.)
 
The above statement of \citename{Greuel:Pfister:1993} leads to the
definition, made in this paper as \dref{2dExcellent}, of an
``excellent quotient,'' which is essentially a universal geometric
quotient $X \to Y$ with the additional property that
$X \cong F \times Y$ as schemes over $Y$, with $F$ another
scheme. Unsurprisingly, an excellent quotient is better than a
geometric quotient. For example, an excellent quotient implies the
existence of a {\em cross section} $Y \to X$ (meaning that the
composition $Y \to X \to Y$ is the identity), and if $X = \Spec(R)$
and $Y = \Spec(R^G)$, then the cross section means that $R^G$ is the
image of a ring map from $R$ to $R$. So the invariant ring tends to
have exceptionally few generators with an easy way of computing them.

This paper goes beyond unipotent groups by considering connected
solvable groups. This is of interest not only because solvable groups
naturally extend the class of unipotent groups, but also because if
$G$ is a connected algebraic group acting on an affine variety $X$ and
$B$ is a Borel subgroup, then $K[X]^G = K[X]^B$ (see
\mycite[Exercise~21.8]{Humphreys}); so computing invariant rings of
connected solvable groups means computing invariant rings of all
connected groups. The main results of the paper, to be found in detail
in \tref{5tMain} and \rref{5rQuasiaffine}, have already been stated in
the abstract above. To put them in the context of the existing
literature, it should be mentioned that already
\mycite[Theorem~10]{Rosenlicht:1956} showed that the quotient
$X \to Y$ by a connected solvable group, viewed as a rational map, has
a cross section. However, working only with rational maps, he did not
consider geometric quotients in the modern setup, and his proof is far
from constructive. \mycite[Theorem~3]{Popov:2016} proved that if a
connected solvable group $G$ acts on an irreducible algebraic variety
$X$ over an algebraically closed field, then $X$ has a $G$-stable
dense open subset that admits an excellent quotient. Again, the result
is not constructive.

Thus the results of this paper extend the earlier results mentioned
above in generality ($X$ need not be integral, the ground ring $K$
need not be a field), in scope (solvable groups instead of unipotent
groups), and because the results are fully
algorithmic. (\mycite{Sancho:2000} presents algorithms for the
additive group and gives ideas towards algorithms for unipotent
groups.)  The result about complete intersections (see
\tref{5tMain}\eqref{5tMainB}) seems to be entirely new.

It might seem that generalizing from unipotent to solvable groups is
meaningless since a solvable group consists of a unipotent group $U$
and a torus on top, for which the invariant theory is easy and
harmless. However, for the group $U$, an excellent quotient (or even a
categorical quotient) only exists after passing to $\Spec(R_f)$
with~$f$ an invariant, and for most choices of such an~$f$, the torus
will not act on $R_f^U$. In fact, the most technically involved part
of this paper is the proof that~$f$ can be chosen as a semi-invariant
of the torus, and that such a choice can be made in the general
situation assumed here and at a low computational cost. \\

The first section of this paper is devoted to actions of the additive
group on an affine scheme. Such actions have been studied in various
papers, e.g. \mycite{Tan:89}, \mycite{essen},
\mycite{Derksen.Kemper06}, \mycite{Freudenburg:2006}, and
\mycite{Tanimoto:2007}. Although the main ideas of the section are
already present in these papers (particularly
in~[\citenumber{Tanimoto:2007}]), none of them reaches the level of
generality we require: a general ring, of any characteristic, as
ground ring, and actions on possibly non-integral schemes. \sref{1sGa}
introduces a variant of the notion of a ``local slice'' and gives a
simplified algorithm for computing it. The main result,
\tref{1tSlice}, is the algebraic way of saying that
$\Spec(R_f) \to \Spec(R_f^\Ga)$ is an excellent quotient.

\sref{2sExcellent} introduces the notion of an excellent quotient and
studies some basic properties: An excellent quotient is a universal
geometric quotient, and excellent quotients can be put on top of each
other along a chain of normal subgroups.

Sections~\ref{3sAmbient} and~\ref{4sUnipotent} are rather technical
and address the question, mentioned above, how a local slice can be
found such that the denominator~$f \in R$ is a semi-invariant. The
setup is that the additive group appears as a normal subgroup of an
ambient group, which in \sref{4sUnipotent} is assumed to be connected
and solvable. More precisely, the ambient group is assumed to be ``in
standard solvable form'' according to \dref{4dSolvable}, a hypothesis
that is automatically satisfied when the ground ring is an
algebraically closed field.

\sref{5sTorus} starts by dealing with actions of the multiplicative
group $\Gm$. The results bear an uncanny resemblance to those about
$\Ga$-actions. Putting all the strands together then yields the main
results of the paper (\tref{5tMain} and \rref{5rQuasiaffine}) and its
algorithmic version (\aref{5aSolvable}). The algorithm has been
implemented in the computer algebra system MAGMA~[\citenumber{magma}],
though not in complete generality. It turns out that the excellent
quotient by a connected solvable group has fibers that are isomorphic,
as a scheme without group structure or group action, to another
connected solvable group.

The final section contains a sort of a converse: If a group action
restricts to an open subset $X_f$ where the orbits are all of the type
described above, then the action is ``essentially solvable''
(\tref{6tConverse}).



\vspace{-1mm}
\subsection*{Acknowledgements}

This article benefited from interesting and helpful conversations with
Stephan Neupert, Vladimir Popov, Hanspeter Kraft, and Igor
Dolgachev. In particular, Vladimir Popov provided a proof of the last
statement of \lref{5lKraft} and then discovered that this had also
been shown by \mycite{Borel:1985}. He also made me aware of his
paper~[\citenumber{Popov:2016}] and pointed out that Theorem~2 in that
paper permits to extend my results to actions on arbitrary irreducible
varieties (see \rref{5rQuasiaffine}). The idea for the proof of
\tref{6tConverse}\eqref{6tConverseB} goes back to Hanspeter Kraft, and
Igor Dolgachev helped me with \exref{2exSL2}. My thanks go to all of
them.

\section{Additive group actions} \label{1sGa}

In this section we consider a morphic action of the additive group
$\Ga = \Spec(K[z])$ over a ring $K$ on an affine scheme $\Spec(R)$,
with $R$ a $K$-algebra. Such an action induces a homomorphism
$\map{\phi}{R}{R[z]}$ of $K$-algebras. If $s \in R$ and
$g := \phi(s)$, then
\begin{equation} \label{1eqGaAction}%
  g(0) = s \quad \text{and} \quad \phi(g(w)) = g(w + z),
\end{equation}
where in the second equality~$\phi$ is applied to the polynomial ring
$R[w]$ coefficient-wise. Let us call $\deg(s) := \deg_z(g)$ the
\df{degree} of~$s$. The invariant ring is $R^\Ga := \ker(\phi - \id)$.
If $g = \sum_{i=0}^d c_i z^i$ with $c_i \in R$, $c_d \ne 0$, it
follows that
\begin{equation} \label{1eqBinom}
  \phi(c_i) = \sum_{j=0}^{d-i} \binom{i + j}{i} c_{i+j} z^j.
\end{equation}
In particular, $c := c_d \in R^\Ga$ is an invariant. As we will see,
throughout the paper $c$ is the invariant that was denoted by~``$f$''
in the abstract and introduction. (Here we need the letter~$f$ for
another purpose.)  In fact, we can form the localization $R_c$ of $R$
with respect to the multiplicative set $\{1,c,c^2,\ldots\}$ and
extend~$\phi$ to a homomorphism $R_c \to R_c[z]$, which we will also
call~$\phi$ and which satisfies~\eqref{1eqGaAction}.

I learned the following argument, leading up to the proof of
\pref{1pDWR}, from \linebreak \mycite{Tanimoto:2007}. It is presented
here for the convenience of the reader and since our situation is
slightly different. Let $a \in R$ be another ring element and set
$f := \phi(a) \in R[z]$.  Since $c \in R^\Ga$ is the highest
coefficient of the above polynomial~$g$, we can perform division with
remainder in $R_c[z]$, which gives
\begin{equation} \label{1eqDWR}%
  f = q g + h
\end{equation}
with $q,h \in R_c[z]$, $\deg_z(h) \le d-1$ and $\deg_z(q) \le \deg(a) - d$
(where we assign the degree~$-\infty$ to the zero polynomial).
Using~\eqref{1eqGaAction}, we obtain
\begin{multline*}
  g(w + z) \bigl(q(w + z) - \phi(q(w))\bigr) + \bigl(h(w + z) -
  \phi(h(w))\bigr) = \\
  g(w + z) q(w + z) + h(w + z) - \phi(g(w)) \phi(q(w)) - \phi(h(w)) =
  f(w + z) - \phi(f(w)) = 0
\end{multline*}
Considering this as an equality of polynomials in $w$ and using the
$w$-degree, we conclude that $q(w+z) = \phi(q(w))$ and
$h(w+z) = \phi(h(w))$. Substituting $w = 0$ yields $\phi(q(0)) = q$
and $\phi(h(0)) = h$, so $\deg(q(0)) \le \deg(a) - d$ and
$\deg(h(0)) \le d-1$. We can write $q(0) = r/c^m$ and $h(0) = b/c^m$ with
$r,b \in R$, choosing the integer~$m$ large enough such that~$r$
and~$b$ have the same degrees as~$q(0)$ and~$h(0)$, respectively. Now
substituting $z = 0$ in~\eqref{1eqDWR} and possibly choosing~$m$ even
larger yields $c^m a = r s + b$. In summary, we obtain the following
``division with remainder principle'' in $R$:

\begin{prop} \label{1pDWR}%
  For $s,a \in R$, let~$c \in R^\Ga$ be the highest coefficient of
  $\phi(s)$. Then there exist $r,b \in R$ and $m \in \NN_0$ such that
  \[
  c^m a = r s + b, \quad \deg(r) \le \deg(a) - \deg(s), \quad
  \text{and} \quad \deg(b) \le \deg(s) - 1.
  \]
  If it is possible to perform addition, multiplication, and zero
  testing of elements of $R$, then~$m$, $r$ and~$b$ can be computed.
\end{prop}



\begin{defi} \label{1dSlice}%
  A \df{local slice} of degree~$d$ with denominator~$c$ is a
  noninvariant $s \in R \setminus R^\Ga$ of degree~$d$ with $c \in
  R^\Ga$ the highest coefficient of $\phi(s)$, such that every $a \in
  R_c$ with $\deg(a) < d$ lies in $R_c^\Ga$.
\end{defi}

The significance of this notion lies in the fact that if $s$ is a
local slice, then $b$ from \pref{1pDWR} is an invariant in $R_c^\Ga$.

\begin{rem} \label{1rSlice}%
  If $s \in R \setminus R^\Ga$ has minimal degree among all
  noninvariants, it is a local slice. This shows the existence of
  local slices if the action is nontrivial. If $R$ is a domain, the
  converse holds. So our definition of a local slice is consistent
  with the one from \mycite{Freudenburg:2006} and
  \mycite{Tanimoto:2007}, who only considered domains.
\end{rem}

If the characteristic of $K$ is~$0$ or a prime and if $R$ is reduced,
it is not hard to see from~\eqref{1eqBinom} that the degree~$d$ of a
local slice must be a power of the characteristic of $K$. In
particular, if $\ch(K) = 0$ and $R$ is reduced, an element is a local
slice if and only if it has degree~$1$. The following example shows
that local slices of degree $> 1$ occur.

\begin{ex} \label{1exSlice}%
  \begin{enumerate}
    \renewcommand{\theenumi}{\arabic{enumi}}
  \item Let $R = K[x,y]$ be a polynomial ring over a field of
    characteristic $p > 0$ and define $\map{\phi}{R}{R[z]}$ by
    \[
    \phi(x) = x + y z + z^p, \ \phi(y) = y.
    \]
    Then $s = x$ is a local slice of degree~$p$.
  \item With $K$ a ring of any characteristic, let $R = K[x,y]/(y^n)$
    with $n \ge 2$ and $\phi(x) = x + \overline{y} z$,
    $\phi(\overline{y}) = \overline{y}$. Then for $0 < d < n$, we see
    that $s = x^d$ is a local slice of degree~$d$ with denominator
    $c = \overline{y}^d$, since $R_c = \{0\}$. Perhaps more
    significantly, all local slices have nilpotent denominators, so
    $R_c$ is always the zero ring. \exend
  \end{enumerate}
  \renewcommand{\exend}{}
\end{ex}

Algorithms for finding a local slice (in the case that $R$ is a
finitely generated domain) were given by \mycite{Sancho:2000} and
\mycite{Tanimoto:2007}. The following algorithm for the same purpose
is simpler, and does not require $R$ to be a domain.

\begin{alg}[Computation of a local slice] \label{1aSlice} \mbox{}%
  \begin{description}
  \item[\bf Input] A nontrivial morphic $\Ga$-action on $\Spec(R)$ for
    a finitely generated $K$-algebra $R = K[a_1 \upto a_n]$, given by
    $\map{\phi}{R}{R[z]}$ as above.  We assume that it is possible to
    perform addition, multiplication, and zero testing of elements of
    $R$.
  \item[\bf Output] A local slice $s \in R$.
  \end{description}
  \begin{enumerate}
    \renewcommand{\theenumi}{\arabic{enumi}}
  \item \label{1aSlice1} For $i = 1 \upto n$, set $b_i = a_i$. Repeat
    steps~\ref{1aSlice2}--\ref{1aSlice3} until all $b_i$ have
    degree~$0$.
  \item \label{1aSlice2} Choose $s \in R$ as a noninvariant
    coefficient of one of the $\phi(b_i)$ such that $d := \deg(s)$
    becomes minimal. One can use~\eqref{1eqBinom} for determining
    $\phi(s)$. If $\ch(K) = 0$, then automatically $d = 1$, so~$s$ is
    a local slice and we are done.
  \item \label{1aSlice3} This step updates the~$b_i$. For
    $i = 1 \upto n$, find elements $r_i, b_i \in R$ and $m_i \in
    \NN_0$ such that
    \begin{equation} \label{1eqCai}%
      c^{m_i} a_i = r_i s + b_i \quad \text{and} \quad \deg(b_i) \le d
      -1
    \end{equation}
    according to \pref{1pDWR}.
  \end{enumerate}
\end{alg}

\begin{proof}[Proof of correctness of \aref{1aSlice}]
  The choice of~$s$ implies $d \le \deg(b_i)$ for all~$i$, so updating
  the $b_i$ in step~\ref{1aSlice3} strictly decreases their degree. So
  the algorithm will reach its termination point where all~$b_i$ have
  degree~$0$.

  Suppose this is the case and let $a \in R_c$ be of degree $< d$.
  There is a nonnegative integer~$k$ such that
  $a = c^{-k} F(a_1 \upto a_n)$ with $F \in K[x_1 \upto x_n]$ a
  polynomial. In the ring $R_c$, $s$ divides
  $F(a_1 \upto a_n) - F(c^{-m_1} b_1 \upto c^{-m_n} b_n)$
  by~\eqref{1eqCai}, so it also divides
  $a - c^{-k} F(c^{-m_1} b_1 \upto c^{-m_n} b_n)$. But this difference
  has degree $< d$, and since the highest coefficient of $\phi(s)$ is
  invertible in $R_c$, this implies that
  $a = c^{-k} F(c^{-m_1} b_1 \upto c^{-m_n} b_n)$, which is an
  invariant in $R_c^\Ga$.
\end{proof}

The following theorem shows why local slices are useful. For example,
by part~\eqref{1tSliceB}, generators of $R_c^\Ga$ can be determined
immediately if a local slice is known. While parts~\eqref{1tSliceA}
and~\eqref{1tSliceB} are essentially well known (at least in more
restricted settings), \eqref{1tSliceC} and~\eqref{1tSliceD} seem to be
entirely new.

\begin{theorem} \label{1tSlice}%
  For a nontrivial action of the additive group $\Ga$ over a ring $K$
  on an affine $K$-scheme $\Spec(R)$, given by a homomorphism
  $\map{\phi}{R}{R[z]}$, let~$s$ be a local slice of degree~$d$ with
  denominator $c \in R^\Ga$.
  \begin{enumerate}
  \item \label{1tSliceA} The homomorphism $R_c^\Ga[x] \to R_c$ sending
    the indeterminate~$x$ to~$s$ is an isomorphism. We write
    $\map{\psi}{R_c}{R_c^\Ga[x]}$ for the inverse isomorphism.
  \item \label{1tSliceB} The composition
    \[
    \pi\mbox{:}\ R_c \xlongrightarrow{\psi} R_c^\Ga[x]
    \xlongrightarrow{x \mapsto 0} R_c^\Ga
    \]
    is a homomorphism of $R_c^\Ga$-algebras with $\ker(\pi) = (s)$. In
    particular,~$\pi$ is surjective. For $a \in R_c$, $\pi(a)$ is
    given by
    \[
    \phi(a) = q \cdot \phi(s) + \pi(a)
    \]
    with $q \in R_c[z]$ (division with remainder).
  \item \label{1tSliceC} The composition
    \[
    R_c \stackrel{\phi}{\longrightarrow} R_c[z]
    \stackrel{\pi}{\longrightarrow} R_c^\Ga[z]
    \]
    (with~$\pi$ applied coefficient-wise) is injective and makes
    $R_c^\Ga[z]$ into an $R_c$-module that is generated by~$d$
    elements. In particular, if $d = 1$, then it is an isomorphism.
  \item \label{1tSliceD} Let $B$ be a ring with a homomorphism
    $R_c^\Ga \to B$. Then
    \[
    (B \otimes_{R_c^\Ga} R)^\Ga = B \otimes 1.
    \]
  \end{enumerate}
\end{theorem}

\begin{proof}
  \begin{enumerate}
  \item[\eqref{1tSliceA}] Let $f \in R_c^\Ga[x]$ with $f(s) = 0$.
    Since $\phi$ is a homomorphism of $R_c^\Ga$-algebras, this implies
    $f(\phi(s)) = 0$, so $f = 0$ since the highest coefficient of
    $g := \phi(s) \in R_c[z]$ is invertible. This shows that the map
    is injective. To prove surjectivity, let $a \in R_c$. Applying
    \pref{1pDWR} to $R_c$ instead of $R$ yields $a = r s + b$ with
    $r,b \in R_c$ such that $\deg(b) < d$ and
    $\deg(r) \le \deg(a) - d < \deg(a)$. By \dref{1dSlice},
    $b \in R_c^\Ga$. Now using induction on $\deg(a)$ yields
    $a = f(s)$ with $f \in R_c^\Ga[x]$.
  \item[\eqref{1tSliceB}] The first statement follows
    from~\eqref{1tSliceA}. For the second statement, observe that the
    map that is claimed to be equal to~$\pi$ is a homomorphism of
    $R_c^\Ga$-algebras, since the remainder from division by $\phi(s)$
    has degree~$0$. Since $R_c = R_c^\Ga[s]$, it suffices to check the
    equality of the maps for $a = s$, which is immediate.
  \item[\eqref{1tSliceC}] Take $a \in R_c$ that is mapped to
    zero. By~\eqref{1tSliceA} we may write $a = f(s)$ with $f \in
    R_c^\Ga[x]$, so
    \[
    0 = \pi\bigl(\phi\bigl(f(s)\bigr)\bigr) = f\bigl(\pi(g)\bigr) =
    f\bigl(\pi(g - s) + \pi(s)\bigr)= f(g - s)
    \]
    since by~\eqref{1eqBinom} all coefficients of~$g - s$ have degree
    $< d$ and are therefore invariants in $R_c^\Ga$. Since
    $g - s \in R_c[z]$ has $z$-degree $d > 0$ and an invertible
    highest coefficient, we obtain $f = 0$. This establishes
    injectivity. The above equality also shows that
    $g - s = \pi(\phi(s))$ lies in the image. So~$z$ satisfies the
    polynomial $(g(x) - s) - (g - s)$, whose coefficients (as a
    polynomial in~$x$) lie in the image. This proves the second
    statement.
  \item[\eqref{1tSliceD}] By~\eqref{1tSliceA}, the element
    $1 \otimes s \in B \otimes_{R_c^\Ga} R := R'$ is algebraically
    independent over $B$, and $R' = B[1 \otimes s]$. By definition,
    $(R')^\Ga = \ker(\phi' - \id)$ with $\map{\phi'}{R'}{R'[z]}$
    obtained by tensoring~$\phi$. Let $a \in R'$ and write
    $a = \sum_{i=0}^k b_k (1 \otimes s)^i$ with $b_i \in B$,
    $b_k \ne 0$. With the given map $\map{\eta}{R_c^\Ga}{B}$ applied
    to $\RR_c^\Ga[z]$ coefficient-wise, we obtain
    \[
    \phi'(a) = \sum_{i=0}^k b_i (1 \otimes g)^i = \sum_{i=0}^k b_i
    \bigl(\eta(g - s) \otimes 1 + 1 \otimes s\bigr)^i.
    \]
    If $k > 0$, the coefficient of $z^{k d}$ of this is $b_k \eta(c)^k
    \otimes 1$, which is nonzero since $\eta(c)$ is invertible in
    $B$. So if $a \in (R')^\Ga$, then $k = 0$ and therefore $a \in B
    \otimes 1$. The reverse inclusion $B \otimes 1 \subseteq (R')^\Ga$
    is clear. \qed
  \end{enumerate}
  \renewcommand{\qed}{}
\end{proof}

\tref{1tSlice} has a geometric interpretation. In fact, the morphism
$\Spec(R_c) \to \Spec(R_c^\Ga)$ induced from the inclusion is a
excellent quotient by $\Ga$ with fibers $\Spec(K[x])$ according to
\dref{2dExcellent} in the following section (see \pref{2pExcellent}).

\section{Excellent quotients} \label{2sExcellent}%

In the following, $S$ is a scheme and all schemes, morphisms and fiber
products will be over $S$ unless stated otherwise.

\newcommand{\act}{\operatorname{act}}
\newcommand{\qu}{\operatorname{quo}}
\newcommand{\pt}{\operatorname{pt}}
\newcommand{\iso}{\operatorname{iso}}
\newcommand{\pr}{\operatorname{pr}}
\newcommand{\sect}{\operatorname{sect}}
\newcommand{\incl}{\operatorname{incl}}
\newcommand{\emb}{\operatorname{emb}}
\newcommand{\conj}{\operatorname{conj}}
\newcommand{\mult}{\operatorname{mult}}
\newcommand{\inv}{\operatorname{inv}}
\newcommand{\sm}[1]{\text{\scriptsize $#1$}}
\newcommand{\OX}[1]{\mathcal{O}_{#1}}
\newcommand{\gs}[1]{\Gamma(#1,\OX{#1})}

\begin{defi} \label{2dExcellent}%
  Let $G$ be a group scheme acting on a scheme $X$ by a morphism
  $\map{\act}{G \times X}{X}$. A morphism $\map{\qu}{X}{Y}$ of schemes
  is called an \df{excellent quotient} (of $X$ by $G$ with fibers $F$)
  if there is a faithfully flat scheme $F$ with a morphism
  $\map{\pt}{S}{F}$ (i.e., an $S$-valued point of $F$) and an
  isomorphism $\map{\iso}{F \times Y}{X}$ of schemes over $Y$ such
  that the following conditions hold.
  \begin{enumerate}
    \renewcommand{\theenumi}{\roman{enumi}}
  \item \label{2dExcellent1} With $\map{\pr_2}{G \times X}{X}$ the
    second projection, the diagram \vspace{-2mm}%
    \DIAGV{60}%
    {G \times X} \n{\Ear{\sm{\act}}} \n{X} \nn%
    {\Sar{\sm{\pr_2}}} \n{} \n{\Sar{\sm{\qu}}} \nn%
    X \n{\Ear{\sm{\qu}}} \n{Y}%
    \diag%
    commutes.
  \item \label{2dExcellent2} The composition
    \[
    G \times Y \xlongrightarrow{(\id_G,\pt,\id_Y)} G \times F \times Y
    \xlongrightarrow{(\id_G,\iso)} G \times X \xlongrightarrow{\act} X
    \]
    is surjective.
  \item \label{2dExcellent3} Let $Y' \to Y$ be a morphism of schemes,
    giving rise to morphisms $\map{\qu'}{X' := X \times_Y Y'}{Y'}$ and
    $\map{\act'}{G \times X'}{X'}$ by base change. Then the map
    $\gs{Y'} \to \gs{X'}$ induced by $\qu'$ has the ring $\gs{X'}^G$
    of $G$-invariant functions as its image, defined (in the usual
    way) as follows: An element $f \in \gs{X'}$, interpreted as an
    element of $\Hom_\ZZ(X',\AA^1)$ with
    $\AA^1 := \AA^1_\ZZ := \Spec(\ZZ[x])$, lies in $\gs{X'}^G$ if and
    only if the diagram \vspace{-2mm}%
    \begin{equation} \label{2eqInvariant}
      \begin{minipage}{0.8\linewidth}
        \DIAGV{60}%
        {G \times X'} \n{\Ear{\sm{\act'}}} \n{X'} \nn%
        {\Sar{\sm{\pr_2}}} \n{} \n{\Sar{\sm{f}}} \nn%
        {X'} \n{\Ear{\sm{f}}} \n{\AA^1}%
        \diag%
      \end{minipage}
    \end{equation}
    commutes.
  \end{enumerate}
\end{defi}

The following remark should provide a better understanding of
\dref{2dExcellent}.

\begin{rem} \label{2rExcellent}
  \begin{enumerate}
  \item \label{2rExcellentC} If $\map{\qu}{X}{Y}$ is an excellent
    quotient, the commutative diagram \vspace{-2mm}%
    \DIAGV{60}%
    {Y} \n{\Ear{\sm{(\pt,\id)}}} \n{F \times Y} \n{\Ear{\sm{\iso}}}
    \n{X} \nn%
    {} \n{} \n{\eseaR{\sm{\id}}} \n{\Sear{\sm{\pr_2}}}
    \n{\saR{\sm{\qu}}} \nn%
    {} \n{} \n{} \n{} \n{Y}%
    \diag%
    shows that the morphism
    $\map{\sect := \iso \circ (\pt,\id_Y)}{Y}{X}$ satisfies
    \begin{equation} \label{2eqSection}%
      \qu \circ \sect = \id_Y,
    \end{equation}
    so it is a \df{cross section} of the quotient. This implies that
    $X \to Y$ is surjective. Condition~\eqref{2dExcellent2} in
    \dref{2dExcellent} demands the surjectivity of
    \begin{equation} \label{2eqSurj}%
      G \times Y \xlongrightarrow{(\id,\sect)} G \times X
      \xlongrightarrow{\act} X.
    \end{equation}
    If $S$ is the spectrum of an algebraically closed field, this
    means that every $G$-orbit in $X$ meets the image of the cross
    section, or, equivalently, that the fibers of the quotient are
    precisely the orbits.
  \item \label{2rExcellentB} In \dref{2dExcellent}, the point~$\pt$
    and the isomorphism~$\iso$ only appear in
    Condition~\eqref{2dExcellent2}. It is not hard to show that if
    this condition holds for some choice of a point and an
    isomorphism, then it holds for all choices.
  \item \label{2rExcellentE} If $\map{\qu}{X}{Y}$ is an excellent
    quotient with fibers $F$, then all fibers of $S$-valued points of $Y$
    are isomorphic to $F$. This follows since $X$ and $F \times Y$ are
    isomorphic as schemes over $Y$.
  \item \label{2rExcellentF} Since $G$ acts on the fiber of every
    $S$-valued point of $Y$, it follows from~\eqref{2rExcellentE} that
    every such point affords a $G$-action on $F$. But these actions
    are in general different, so there is usually no $G$-action on $F$
    that makes the isomorphism $F \times Y \to X$ into a
    $G$-isomorphism.
  \item \label{2rExcellentD} A base change of an excellent quotient is
    again an excellent quotient (by the same group and with the same
    fibers). This follows directly from the definition and the fact
    that surjectivity is stable under base change (see
    \mycite[Proposition~4.32]{Goertz:Wedhorn:2010}).
  \item \label{2rExcellentG} Let $\map{\qu}{X}{Y}$ be an excellent
    quotient by a group $G$, and let $U \subseteq X$ be a $G$-stable
    open subscheme. With $\map{\sec}{Y}{X}$ the cross section,
    $V := \sect^{-1}(U) \subseteq Y$ is an open subscheme, and it is
    not hard to see that $U = \qu^{-1}(V)$.
    Therefore the restriction $\map{\qu|_U}{U}{V}$ arises from
    $X \to Y$ by base change. By~\eqref{2rExcellentD}, it is an
    excellent quotient of $U$ by $G$. \remend
  \end{enumerate}
  \renewcommand{\remend}{}
\end{rem}

The purpose of the following examples is to show how strong the notion
of an excellent quotient is.

\begin{ex} \label{2exSL2}%
  In this example we assume that $K$ is an algebraically closed field
  of characteristic $\ne 2$. Consider the action of $G = \PGL_2$ on
  $\SL_2$ by conjugation. The only invariant is given by the trace,
  and it is well known that restricting to the matrices with distinct
  eigenvalues gives a geometric quotient
  \[
  X := \bigl\{A \in \SL_2 \mid \tr(A) \ne \pm 2\bigr\} \to Y := K
  \setminus \{\pm 2\}.
  \]
  The quotient has a cross section, given by mapping $a \in Y$ to the
  matrix
  $\left(\begin{smallmatrix} 0 & -1 \\ 1 &
      a \end{smallmatrix}\right)$.
  Next we see that all fibers are isomorphic. Indeed, the fiber of
  $a \ne \pm 2$ consists of the matrices
  $\left(\begin{smallmatrix} x & y \\ z & a -
      x \end{smallmatrix}\right)$ satisfying
  \[
  0 = x^2 - a x + y z + 1 = \left(\frac{2 x - a}{2}\right)^2 + y z +
  \frac{4 - a^2}{4} = \frac{4 - a^2}{4}
  \left(\left(\frac{2x-a}{\sqrt{4 - a^2}}\right)^2 + \frac{4y}{4 -
      a^2} \cdot z + 1\right).
  \]
  The last form of the equation shows that all fibers are isomorphic
  to the $0$-fiber given by $x^2 + y z + 1$. So the quotient has
  extremely good properties, but we claim that is is not excellent.

  In fact, if it were excellent, then $X$ would be isomorphic, as a
  scheme over $Y$, to $F \times Y$, with $F$ the surface given by
  $x^2 + y z + 1$. Since $X$ is given by the equation
  $x^2 - a x + y z + 1$ (see above), a $Y$-isomorphism
  $F \times Y \cong X$ would imply that the surfaces {\em over the
    rational function field} $K(t)$ given by $x^2 + y z + 1$ and by
  $x^2 - t x + y z + 1$ are isomorphic. But they are not, and the
  reason for this was explained to me by Igor Dolgachev. First,
  homogenizing the equations defines two projective quadrics in
  $\PP^3_L$, which are not isomorphic since their discriminants are in
  different square classes. Second (and this is the hard part), it
  follows from a result by \mycite[Theorem~6]{Gizatullin:Danilov} that
  also the original affine surfaces cannot be isomorphic over $K(t)$.

  If we consider the action of $\SL_2$ on quadratic binary forms, we
  also get a geometric quotient on a subset which has a cross section,
  and all fibers are isomorphic, but the quotient is not
  excellent. The proof is virtually the same as above.
\end{ex}

\begin{ex} \label{2exSO2}%
  Let $K$ be a field in which~$-1$ is not a square. Consider the
  natural action of
  $G = \SO_2(K) = \{A \in \GL_2(K) \mid A^T A = I_2,\ \det(A) = 1\}$
  on $X = \AA^2_K$. The invariant ring $K[X]^G$ is known to be
  generated by $x_1^2 + x_2^2$, which defines a quotient
  $X \to \AA^1_K$.  But this has no cross section, even after choosing
  a nonempty open subset $Y \subseteq \AA^1_K$ and restricting to its
  preimage $X$. In fact, giving such a cross section is equivalent to
  giving polynomials $f,g,h \in K[x]$ such that $f^2 + g^2 = x h^2$
  with $h^2 \ne 0$. It is not hard to see that this is possible (if
  and) only if~$-1$ is a square in $K$.
  It follows that the quotient $X \to Y$ is not
  excellent.
\end{ex}

The following proposition deals with the affine case and shows how
Condition~\eqref{2dExcellent3} of \dref{2dExcellent} can be verified.

\begin{prop} \label{2pExcellent}%
  Assume the situation of \dref{2dExcellent}. If $X = \Spec(R)$ and
  $Y = \Spec(B)$ with $B \subseteq R$ rings, then
  Condition~\eqref{2dExcellent1} of \dref{2dExcellent} is equivalent
  to $B \subseteq R^G$, and, given~\eqref{2dExcellent1},
  Condition~\eqref{2dExcellent3} is equivalent to the following:
  \begin{enumerate}
    \renewcommand{\theenumi}{\ref{2dExcellent3}'}
  \item \label{2dExcellent3p} For every homomorphism $B \to B'$ of
    rings we have
    \[
    (B' \otimes_B R)^G \subseteq B' \otimes 1.
    \]
  \end{enumerate}
  In particular, in the situation of \tref{1tSlice}, the morphism
  $\Spec(R_c) \to \Spec(R_c^\Ga)$ is an excellent quotient by $\Ga$
  with fibers $\AA^1_K$.
\end{prop}

\begin{proof}
  First assume that Condition~\eqref{2dExcellent1} of
  \dref{2dExcellent} holds, and let
  $f \in B = \gs{Y} = \Hom_\ZZ(Y,\AA^1)$. Viewing~$f$ as an element of
  $R$ means to consider $\map{f \circ \qu}{X}{\AA^1}$, which lies in
  $\gs{X}^G = R^G$ by~\eqref{2dExcellent1}. Conversely, assume that
  $f \circ \qu$ lies in $\gs{X}^G$ for all $f \in
  \Hom_\ZZ(Y,\AA^1)$. So
  \[
  f \circ \qu \circ \act = f \circ \qu \circ \pr_2.
  \]
  The morphisms $\qu \circ \act$ and $\qu \circ \pr_2$ are both into
  the affine scheme $Y$, so they are given by homomorphisms
  $\map{\phi_1,\phi_2}{B}{\gs{G \times X}}$ (see
  \mycite[Proposition~3.4]{Goertz:Wedhorn:2010}), and, by the above
  equality, for every homomorphism $\map{\psi}{\ZZ[x]}{B}$ we have
  $\phi_1 \circ \psi = \phi_2 \circ \psi$. This implies
  $\phi_1 = \phi_2$ and therefore~\eqref{2dExcellent1}.
  
  Now assume that Condition~\eqref{2dExcellent3} of \dref{2dExcellent}
  holds, and let $B \to B'$ be a homomorphism of rings, inducing a
  morphism $Y' := \Spec(B') \to Y$. The map $\map{\qu'}{X'}{Y'}$
  induces the homomorphism $\mapl{\phi}{B' = \gs{Y'}}{\gs{X'} = B'
    \otimes_B R}{b'}{b' \otimes 1}$, and we have
  \[
  (B' \otimes_B R)^G = \gs{X'}^G \underset{\eqref{2dExcellent3}}{=}
  \phi\bigl(\gs{Y'}\bigr) = B' \otimes 1.
  \]
  so~\eqref{2dExcellent3p} follows.

  Conversely assume~\eqref{2dExcellent3p} and let $Y' \to Y$ be a
  morphism of schemes. Since $\qu' \circ \act' = \qu' \circ \pr_2$,
  the image of the map $\map{\phi}{\gs{Y'}}{\gs{X'}}$ induced by
  $\qu'$ is contained in $\gs{X'}^G$.  It remains to show the reverse
  inclusion. Let $V \subseteq Y'$ be an open subset and let
  $U := (\qu')^{-1}(V) \subseteq X'$ be its inverse image. Since all
  squares in the diagram
  \begin{equation} \label{2eqCart}%
  \begin{minipage}{0.8\linewidth}
    \DIAGV{60}%
    {U} \n{\Ear{\sm{\incl}}} \n{X'} \n{\ear} \n{X}
    \n{\Ear{\sm{\iso^{-1}}}} \n{F \times Y} \n{\ear} \n{F} \nn%
    {\Sar{\sm{\qu'|_U}}} \n{\square} \n{\Sar{\sm{\qu'}}} \n{\square}
    \n{\Sar{\sm{\qu}}} \n{\square} \n{\Sar{\sm{\pr_2}}} \n{\square}
    \n{\sar} \nn%
    {V} \n{\Ear{\sm{\incl}}} \n{Y'} \n{\ear} \n{Y} \n{\Ear{\sm{\id}}}
    \n{Y} \n{\ear} \n{S}%
    \diag%
  \end{minipage}
  \end{equation}
  are cartesian, so is the outer rectangle (see
  \mycite[Proposition~4.16]{Goertz:Wedhorn:2010}). Therefore $\qu'|_U$
  is faithfully flat (see~[\citenumber{Goertz:Wedhorn:2010},
  Remark~14.8]), and it follows by \mycite[Corollaire~2.2.8]{EGA4}
  that the map $\gs{V} \to \gs{U}$ induced by it is injective%
  . If $V$ is affine, say
  $V = \Spec(B')$, then $U = V \times_Y X = \Spec(B' \otimes_B R)$,
  and the map $B' = \gs{V} \to \gs{U} = B' \otimes_B R$ induced by
  $\qu'|_U$ is given by $b' \to b' \otimes 1$. Let $f \in \gs{X'}^G$,
  viewed as a morphism $X' \to \AA^1$. Then
  \[
  f|_U \in \Hom_\ZZ(U,\AA^1)^G = (B' \otimes_B R)^G
  \underset{\eqref{2dExcellent3p}}{\subseteq} B' \otimes 1,
  \]
  so there exists $g_V \in B' = \Hom_\ZZ(V,\AA^1)$ with $g_V \circ
  \qu'|_U = f|_U$. By the injectiveness of $\gs{V} \to \gs{U}$, $g_V$
  is uniquely determined. Now it follows from the sheaf property of
  $\OX{Y'}$ that there exists $g \in \Hom_\ZZ(Y',\AA^1)$ with $g \circ
  \qu' = f$, i.e., $f = \phi(g)$. This completes the proof of the
  equivalence.

  For the last statement of the proposition, observe that $\AA^1_K$ is
  faithfully flat over $\Spec(K)$, and that all conditions from
  \dref{2dExcellent} follow directly from \tref{1tSlice} and from this
  proposition.
\end{proof}

We will now prove that an excellent quotient is a geometric
quotient. Let us recall this notion. According to
\mycite[Definition~0.6]{MFK}, a morphism $\map{\qu}{X}{Y}$ of schemes
(where $X$ has an action of a group scheme $G$) is a \df{geometric
  quotient} if:

\begin{enumerate}
  \renewcommand{\theenumi}{g$_{\arabic{enumi}}$}
\item \label{g1} Condition~\eqref{2dExcellent1} from
  \dref{2dExcellent} holds.
\item \label{g2} The morphism $\map{(\act,\pr_2)}{G \times X}{X
    \times_Y X}$ is surjective. (If $S$ is the spectrum of an
  algebraically closed field, this means that the fibers of $\qu$ are
  the $G$-orbits.)
\item \label{g3} The morphism $\qu$ is \df{submersive}, i.e., it is
  surjective and a subset $V \subseteq Y$ is open if its preimage
  $\qu^{-1}(V) \subseteq X$ is open.
\item \label{g4} Condition~\eqref{2dExcellent3} from
  \dref{2dExcellent} holds for the case that $Y' = V \subseteq Y$ is
  an open subset of $Y$.
\end{enumerate}

By~[\citenumber{MFK}, Definition~0.7], the morphism is called a
\df{universal geometric quotient} if for every morphism $Y' \to Y$ of
schemes, the morphism $\map{\qu'}{X \times_Y Y'}{Y'}$ obtained by base
change is a geometric quotient. Recall that a geometric quotient is
always a categorical quotient. Therefore the following result implies
that if an excellent quotient $X \to Y$ exists, it is unique up to
isomorphism. However, the cross section $Y \to X$ (see
\rref{2rExcellent}\eqref{2rExcellentC}) is in general not unique.

\begin{theorem} \label{2tGeo}%
  Let $\map{\qu}{X}{Y}$ be an excellent quotient by a group scheme $G$
  with fibers $F$. Then it is a faithfully flat universal geometric
  quotient.
\end{theorem}

\begin{proof}
  The cartesian diagram~\eqref{2eqCart} (without the first two
  columns) and the argument after it show that $X \to Y$ is faithfully
  flat. Since by \rref{2rExcellent}\eqref{2rExcellentD} excellent
  quotients are stable under base change, we only need to show that
  $X \to Y$ is a geometric quotient. The conditions~\eqref{g1}
  and~\eqref{g4} are immediate, and~\eqref{g3} follows since $X \to Y$
  has a cross section by \rref{2rExcellent}\eqref{2rExcellentC}.

  It remains to prove the condition~\eqref{g2}. We will establish the
  surjectivity of $\map{(\act,\pr_2)}{G \times X}{X \times_Y X}$ by
  proving surjectivity on the geometric points with values in fields
  $L$ (see \mycite[Proposition~4.8]{Goertz:Wedhorn:2010}). Fixing $L$
  with a morphism $\Spec(L) \to S$, we have a functor from the
  category of $S$-schemes to the category of sets, which assigns to an
  $S$-scheme $A$ the set $\widehat{A} := \Hom_S(\Spec(L),A)$, and to a
  morphism $\map{f}{A}{B}$ of $S$-schemes the map
  $\mapl{\widehat{f}}{\widehat{A}}{\widehat{B}}{z}{f \circ z}$. It is
  easy to see that
  $\map{(\widehat{\pr}_1,\widehat{\pr}_2)}{\widehat{A \times
      B}}{\widehat{A} \times \widehat{B}}$
  is a bijection between the fiber product and the cartesian product,
  and, more generally, for an $S$-scheme $C$ and morphisms
  $\map{f}{A}{C}$, $\map{g}{B}{C}$ the map
  \[
  \map{(\widehat{\pr}_1,\widehat{\pr}_2)}{\widehat{A \times_C
      B}}{\bigl\{(x,y) \in \widehat{A} \times \widehat{B} \mid
    \widehat{f}(x) = \widehat{g}(y)\bigr\} =: \widehat{A}
    \times_{\widehat{C}} \widehat{A}}
  \]
  is a bijection (see \mycite[(3.4.3.2)]{EGA1}). In particular,
  $\widehat{G}$ is a group acting on $\widehat{X}$. To prove the
  surjectivity, take an arbitrary morphism $\Spec(K) \to X \times_Y X$
  with $K$ a field. For a field extension $L$ this yields morphisms
  $\Spec(L) \to X \times_Y X$, and, by composition $\Spec(L) \to S$.
  By the above, we receive a pair
  $(x_1,x_2) \in \widehat{X} \times_{\widehat{Y}} \widehat{X}$. The
  claimed surjectivity will follow if we can show that $(x_1,x_2)$
  corresponds to a point in the image of
  $\map{(\widehat{\act},\widehat{\pr}_2)}{\widehat{G \times
      X}}{\widehat{X \times _Y X}}$.
  Using the surjectivity of~\eqref{2eqSurj} and Proposition~4.8
  from~[\citenumber{Goertz:Wedhorn:2010}], we can choose $L$ large
  enough such that there exist $g_1,g_2 \in \widehat{G}$ and
  $y_1,y_1 \in \widehat{Y}$ such that
  $x_i = g_i\bigl(\widehat{\sect}(y_i)\bigr)$.
  We have
  \[
  y_i \underset{\eqref{2eqSection}}{=}
  \widehat{\qu}\bigl(\widehat{\sect}(y_i)\bigr)
  \underset{\eqref{2dExcellent1}}{=}
  \widehat{\qu}\left(g_i\bigl(\widehat{\sect}(y_i)\bigr)\right) =
  \widehat{\qu}(x_i),
  \]
  which with
  $(x_1,x_2) \in \widehat{X} \times_{\widehat{Y}} \widehat{X}$ implies
  $y_1 = y_2$. Therefore $x_1 = g_1 g_2^{-1} (x_2)$, so indeed
  $(x_1,x_2)$ corresponds to a point in the image of
  $\map{(\widehat{\act},\widehat{\pr}_2)}{\widehat{G \times
      X}}{\widehat{X \times _Y X}}$.
\end{proof}

It is hardly surprising that the converse of \tref{2tGeo} does not
hold. The following example illustrates this.

\begin{ex} \label{2exCounter}%
  Let the cyclic group $G$ of order~$2$ act on the ring
  $R = K[x,x^{-1}]$ of Laurent polynomials over an integral domain in
  which~$2$ is invertible by mapping~$x$ to $-x$. Then
  $R^G = K[x^2,x^{-2}]$, and $X := \Spec(R) \to \Spec(R^G) =: Y$ is a
  faithfully flat universal geometric quotient with all fibers of
  $K$-points isomorphic to $F := \Spec(K[x]/(x^2 - 1))$. But $X \to Y$
  is not an excellent quotient since $X$ is irreducible, but
  $F \times Y$ is not.

  In fact, for finite group actions the quotient usually has no cross
  section $Y \to X$.
\end{ex}

\begin{rem} \label{2rGPS}%
  As mentioned in the introduction, the papers by
  \mycite{Greuel:Pfister:1993} and \mycite{Sancho:2000} both revolve
  around geometric quotients. In both papers, it appears that the
  following argument is used (see~[\citenumber{Greuel:Pfister:1993},
  Proof of Proposition~1.6] and~[\citenumber{Sancho:2000}, last
  statement of Proposition~2.3]): If $G = \Ga$ (or a connected
  unipotent group) acts on a ring $R$ and $R$ is purely transcendental
  over $R^G$, then $\Spec(R) \to \Spec(R^G)$ is a geometric
  quotient. But this is not true in general: Consider the $\Ga$-action
  on the polynomial ring $R = K[x_1,x_2]$ given by mapping~$x_1$ to
  $x_1 + x_2 z$ and fixing~$x_2$. Then $R^G = K[x_2]$, but the
  quotient is not geometric since fiber $x_2 = 0$ consists of
  $0$-dimensional orbits. So the proofs
  in~[\citenumber{Greuel:Pfister:1993}] and~[\citenumber{Sancho:2000}]
  seem to have a gap. But the statements are correct, the missing link
  being provided by \tref{1tSlice}\eqref{1tSliceC} of this paper.
\end{rem}

The following lemma, which we will need later, deals with invariant
fields and geometric quotients, but not with excellent quotients. This
may be a good place to prove it. Although the lemma is almost
certainly well-known, I could not find it in the literature.

\begin{lemma} \label{2lField}%
  Let $X \to Y$ be a geometric quotient by a group scheme $G$, with
  $X$ an integral scheme. Then $K(X)^G = K(Y)$.
\end{lemma}

\begin{proof}
  We view elements of the function field $K(X)$ as morphisms
  $\map{f}{U}{\AA^1}$, where $U \subseteq X$ is the domain of
  definition of the rational function $X \dashrightarrow \AA^1$
  represented by~$f$ (see \mycite[page~235]{Goertz:Wedhorn:2010}). The
  elements of $K(X)^G$ are those where $U$ is $G$-stable and the
  diagram~\eqref{2eqInvariant} (with $X'$ replaced by $U$ and $\qu'$
  by $\qu|_U$) commutes.

  For an element of $K(Y)$, given by $\map{g}{V}{\AA^1}$, the
  property~\eqref{g1} of the geometric quotient implies that the
  composition $U := \qu^{-1}(V) \xlongrightarrow{\qu} V
  \xlongrightarrow{g} \AA^1$ defines an element of $K(X)^G$, so we
  obtain an embedding $K(Y) \subseteq K(X)^G$. To prove equality, take
  an element of $K(X)^G$, given by $\map{f}{U}{\AA^1}$. Let $V =
  \qu(U)$ be the image of $U$ in $Y$. Since $U$ is $G$-stable, it
  follows from~\eqref{g2} that $\qu^{-1}(V) = U$ (see the argument in
  the proof of remark~(4) in \mycite[page~6]{MFK}). By~\eqref{g3}, $V$
  is open, and now~\eqref{g4} implies that~$f$ lies in $K(Y)$.
\end{proof}

To be able to deal with a solvable group by iterating over a chain of
subgroups, we need that ``putting together'' excellent quotients along
a subgroup chain yields an excellent quotient. This is the contents of
the following result.

\begin{theorem} \label{2tSubgroup}%
  Let $G$ be a group scheme acting on a scheme $X$ by a morphism
  $\map{\act_X}{G \times X}{X}$. Let $H \subseteq G$ be a normal
  subgroup scheme and let $\map{\qu_1}{X}{Y}$ be an excellent quotient
  by $H$ with fibers $F_1$. Then there exists a unique $G$-action on
  $Y$ such that the diagram
  \begin{equation} \label{2eqActY}%
    \begin{minipage}{0.8\linewidth}
      \DIAGV{60}%
      {G \times X} \n{\Ear{\sm{\act_X}}} \n{X} \nn%
      {\Sar{\sm{(\id,\qu_1)}}} \n{} \n{\Sar{\sm{\qu_1}}} \nn%
      {G \times Y} \n{\Ear{\sm{\act_Y}}} \n{Y}%
      \diag%
    \end{minipage}
  \end{equation}
  commutes. Suppose there is an excellent quotient $\map{\qu_2}{Y}{Z}$
  by $G$ with fibers $F_2$. Then $\map{\qu_2 \circ \qu_1}{X}{Z}$ is an
  excellent quotient of $X$ by $G$ with fibers $F_1 \times F_2$.
\end{theorem}

\begin{proof}
  By \tref{2tGeo} $\qu_1$ is a universal geometric quotient, so by
  \mycite[Proposition~0.1]{MFK} it is a universal categorical
  quotient. In particular, $G \times X \to G \times Y$ is a
  categorical quotient by $H$, with $H$ acting trivially on $G$. We
  leave it to the
  reader to check, using the normality of $H$, that the diagram
  \vspace{-3mm}%
  \DIAGV{56}%
  {H \times G \times X} \n{} \n{\Earv{\sm{(\id_G,\act_H)}}{75}} \n{}
  \n{G \times X} \nn%
  {} \n{} \n{} \n{} \n{\Sar{\sm{\act_X}}} \nn%
  {\Sarv{\sm{(\pr_2,\pr_3)}}{130}} \n{} \n{} \n{} \n{X} \nn%
  {} \n{} \n{} \n{} \n{\Sar{\sm{\qu_1}}} \nn%
  {G \times X} \n{\Ear{\sm{\act_X}}} \n{X} \n{\Ear{\sm{\qu_1}}} \n{Y}%
  \diag%
  (with $\act_H$ standing for the $H$-action on $X$) commutes. The
  universal property of $G \times X \to G \times Y$ now yields a
  unique morphism $\map{\act_Y}{G \times Y}{Y}$ such
  that~\eqref{2eqActY} commutes. By a further diagram chase, one
  checks that $\act_Y$ defines an action.

  To show that $\qu_2 \circ \qu_1$ is an excellent quotient, we first
  remark that $F_1 \times F_2$ is faithfully flat (see
  \mycite[Remark~14.8]{Goertz:Wedhorn:2010}). From the given
  $S$-valued points $\map{\pt_i}{S}{F_i}$ we form the composition
  $\pt\mbox{:} \ S \xlongrightarrow{\pt_2} F_2 = S \times F_2
  \xlongrightarrow{(\pt_1,\id)} F_1 \times F_2$. With
  $\map{\iso_1}{F_1 \times Y}{X}$ and $\map{\iso_2}{F_2 \times Z}{Y}$
  the given isomorphisms, the diagram \vspace{-3mm}%
  \DIAGV{60}%
  {F_1 \times F_2 \times Z} \n{}
  \n{\Earv{\sm{(\id_{F_1},\iso_2)}}{70}} \n{} \n{F_1 \times Y} \n{}
  \n{\Earv{\sm{\iso_1}}{80}} \n{} \n{X} \nn%
  {} \n{\seaR{\sm{(\pr_1,\pr_3)}}} \n{} \n{\Swar{\sm{(\id,\qu_2)}}}
  \n{} \n{\Sear{\sm{\pr_2}}} \n{} \n{\swaR{\sm{\qu_1}}} \nn%
  {} \n{} \n{F_1 \times Z} \n{} \n{} \n{} \n{Y} \nn%
  {} \n{} \n{} \n{\seaR{\sm{\pr_2}}} \n{} \n{\swaR{\sm{\qu_2}}} \nn%
  {} \n{} \n{} \n{} \n{Z} \diag%
  commutes, so its upper row defines a $Z$-isomorphism $\map{\iso}{F_1
    \times F_2 \times Z}{X}$. We go on by
  proving~\eqref{2dExcellent1}--\eqref{2dExcellent3} from
  \dref{2dExcellent}.
  \begin{enumerate}
  \item[\eqref{2dExcellent1}] This follows since the diagram
    \vspace{-2mm}%
    \DIAGV{60}%
    {G \times X} \n{} \n{\Earv{\sm{\act_X}}{100}} \n{} \n{X} \nn%
    {\Sar{\sm{\pr_2}}} \n{\Searv{\sm{(\id,\qu_1)}}{47}} \n{} \n{}
    \n{\Sar{\sm{\qu_1}}} \nn%
    {X} \n{} \n{G \times Y} \n{\Ear{\sm{\act_Y}}} \n{Y} \nn%
    {} \n{\seaR{\sm{\qu_1}}} \n{\Sar{\sm{\pr_2}}} \n{}
    \n{\Sarv{\sm{\qu_2}}{48}} \nn%
    {} \n{} \n{Y} \n{\Ear{\sm{\qu_2}}} \n{Z}%
    \diag%
    commutes.
  \item[\eqref{2dExcellent2}] By
    \rref{2rExcellent}\eqref{2rExcellentC} the $\qu_i$ have cross
    sections $\map{\sec_1 = \iso_1 \circ (\pt_1,\id_Y)}{Y}{X}$ and
    $\map{\sec_2 = \iso_2 \circ (\pt_2,\id_Z)}{Z}{Y}$. The commutative
    diagram \vspace{-2mm}%
    \DIAGV{60}%
    {Z} \n{} \n{\Earv{\sm{(\pt_2,\id)}}{100}} \n{} \n{F_2 \times Z}
    \n{} \n{\Earv{\sm{\iso_2}}{90}} \n{} \n{Y} \nn%
    {} \n{} \n{\eseaR{\sm{(\pt,\id)}}} \n{} \n{\saR{\sm{(\pt_1,\id)}}}
    \n{} \n{} \n{} \n{\Sar{\sm{(\pt_1,\id)}}} \nn%
    {} \n{} \n{} \n{} \n{F_1 \times F_2 \times Z} \n{}
    \n{\Earv{\sm{(\id_{F_1},\iso_2)}}{65}} \n{} \n{F_1 \times Y} \nn%
    {} \n{} \n{} \n{} \n{} \n{} \n{\eseaR{\sm{\iso}}} \n{}
    \n{\Sar{\sm{\iso_1}}} \nn%
    {} \n{} \n{} \n{} \n{} \n{} \n{} \n{} \n{X}%
    \diag%
    shows that $\sect:= \iso \circ (\pt,\id_Z)$ is the cross section
    of $\qu := \qu_2 \circ \qu_1$. We know that
    $H \times Y \xlongrightarrow{(\id,\sect_1)} H \times X
    \xlongrightarrow{\act_X} X$
    and
    $G \times Z \xlongrightarrow{(\id,\sect_2)} G \times Y
    \xlongrightarrow{\act_Y} Y$
    are surjective, and will deduce that
    $G \times Z \xlongrightarrow{(\id,\sect)} G \times X
    \xlongrightarrow{\act_X} X$
    is surjective. As in the proof of \tref{2tGeo} we will use
    Proposition~4.8 from \mycite{Goertz:Wedhorn:2010}, and write
    $\widehat{X} = \Hom_S(\Spec(L),X)$ for $L$ a field, and so on. Let
    $\Spec(K) \to X$ be a $K$-geometric point (with $K$ a field), and
    let $x \in \widehat{X}$ be the point obtained by composing with
    $\Spec(L) \to \Spec(K)$ for a field extension $L$. Set
    $y := \widehat{\qu}_1(x) \in \widehat{Y}$. Choosing $L$ large
    enough, we obtain $g \in \widehat{G}$ and $z \in \widehat{Z}$ with
    $g\bigl(\widehat{\sect}_2(z)\bigr) = y$, and $h \in \widehat{H}$,
    $y' \in \widehat{Y}$ with
    $h\bigl(\widehat{\sect}_1(y')\bigr) = x$. It follows that
    \begin{equation} \label{2eqYys}%
      y = \widehat{\qu}_1(x) \underset{\eqref{2dExcellent1}}{=}
      \widehat{\qu}_1\bigl(\widehat{\sect}_1(y')\bigr) = y'
    \end{equation}
    Set $x' := g\bigl(\widehat{\sect}(z)\bigr)$. Then
    \begin{equation} \label{2eqQuoX}%
      \widehat{\qu}_1(x') \underset{\eqref{2eqActY}}{=}
      g\left(\widehat{\qu}_1\bigl(\widehat{\sect}_1
        (\widehat{\sect}_2(z))\bigr)\right) =
      g\bigl(\widehat{\sect}_2(z)\bigr) = y.
    \end{equation}
    Enlarging $L$ again we obtain $\widetilde{h} \in \widehat{H}$ and
    $\widetilde{y} \in \widehat{Y}$ such that $x' =
    \widetilde{h}\bigl(\widehat{\sect}_1(\widetilde{y})\bigr)$. It
    follows that
    \[
    \widetilde{y} =
    \widehat{\qu}_1\bigl(\widehat{\sect}_1(\widetilde{y})\bigr)
    \underset{\eqref{2dExcellent1}}{=} \widehat{\qu}_1(x')
    \underset{\eqref{2eqQuoX}}{=} y,
    \]
    so $\widetilde{h}^{-1}(x') = \widehat{\sect}_1(y)$, and we obtain
    \[
    (h \widetilde{h}^{-1} g)\bigl(\widehat{\sect}(z)\bigr) =
    h\bigl(\widetilde{h}^{-1}(x')\bigr) =
    h\bigl(\widehat{\sect}_1(y)\bigr) \underset{\eqref{2eqYys}}{=} x.
    \]
    This shows that $G \times Z \xlongrightarrow{(\id,\sect)} G \times
    X \xlongrightarrow{\act_X} X$ is surjective, as claimed.
  \item[\eqref{2dExcellent3}] Let $Z' \to Z$ be a morphism of schemes,
    and set $Y' = Y \times_Z Z'$ and $X' = X \times_Y Y'$. With
    $\qu_i'$ obtained by base change, the diagram%
    \DIAGV{55}%
    {X'} \n{\Ear{\sm{\qu_1'}}} \n{Y'} \n{\Ear{\sm{\qu_2'}}} \n{Z'}
    \nn%
    {\sar} \n{\square} \n{\sar} \n{\square} \n{\sar} \nn%
    {X} \n{\Ear{\sm{\qu_1}}} \n{Y} \n{\Ear{\sm{\qu_2}}} \n{Z}%
    \diag%
    is cartesian, so $X' = X \times_Z Z'$ and
    $\qu' = \qu_2' \circ \qu_1'$. By
    \rref{2rExcellent}\eqref{2rExcellentD}, the $\qu_i'$ are excellent
    quotients, and in order to prove~\eqref{2dExcellent3} for $\qu'$
    we may replace $X'$, $Y'$ and $Z'$ by the original $X$, $Y$ and
    $Z$. The assertion $\gs{X}^G = \gs{Z}$ would be trivial if $G$
    were a group. But since it is a group scheme we need to do more.

    It follows from~\eqref{2dExcellent1} that the map $\gs{Z} \to
    \gs{X}$ has its image inside $\gs{X}^G$. For the converse, take $f
    \in \gs{X}^G = \Hom_\ZZ(X,\AA^1)^G$, so the
    diagram~\eqref{2eqInvariant} (with $X'$ replaced by $X$)
    commutes. By restricting the action to $H$ and since $X
    \xlongrightarrow{\qu_1} Y$ is an excellent quotient by $H$ we
    obtain $g \in \Hom_\ZZ(Y,\AA^1)$ with $f = g \circ \qu_1$. We
    claim that $g \in \gs{Y}^G$. The diagram \vspace{-2mm}%
    \DIAGV{48}%
    {G \times Y} \n{} \n{} \n{\Earv{\sm{\act_Y}}{190}} \n{} \n{} \n{Y}
    \nn%
    {} \n{\Nwarv{\sm{(\id,\qu_1)}}{45}} \n{} \n{} \n{}
    \n{\Near{\sm{\qu_1}}} \nn%
    {} \n{} \n{G \times X} \n{\Ear{\sm{\act_X}}} \n{X} \nn%
    {\Sarv{\sm{\pr_2}}{200}} \n{} \n{\Sar{\sm{\pr_2}}} \n{} \n{} \n{}
    \n{\saRv{\sm{g}}{200}} \nn%
    {} \n{} \n{X} \n{} \n{} \n{\Ssear{\sm{f}}} \nn%
    {} \n{\swaR{\sm{\qu_1}}} \n{} \n{} \n{\eseaR{\sm{f}}} \nn%
    {Y} \n{} \n{} \n{\Earv{\sm{g}}{200}} \n{} \n{} \n{\AA^1}%
    \diag%
    commutes. We need to show that $g \circ \act_Y = g \circ
    \pr_2$. Both functions are elements of $\gs{G \times Y}$, and the
    diagram shows that mapping them into $\gs{G \times X}$ yields the
    same element. But since $G \times X \to G \times Y$ is faithfully
    flat, the map $\gs{G \times Y} \to \gs{G \times X}$ is injective
    (see the argument made after~\eqref{2eqCart}). So $g \in
    \gs{Y}^G$, and since $\map{\qu_2}{Y}{Z}$ is a excellent quotient
    by $G$ it follows that there is $h \in \Hom_Z(Z,\AA^1)$ with $g =
    h \circ \qu_2$. So $f = h \circ \qu$, and the proof is
    complete. \qed
  \end{enumerate}
  \renewcommand{\qed}{}
\end{proof}

\section{The additive group as a normal subgroup}
\label{3sAmbient}%

If the additive group $\Ga$ acts on an affine scheme $\Spec(R)$, we
know from \tref{1tSlice} that by choosing a local slice with
denominator~$c$ one obtains an excellent quotient
$\Spec(R_c) \to \Spec(R_c^\Ga)$. Now we assume that $\Ga$ appears as a
normal subgroup in a connected solvable group $G$, and wish to build
an excellent quotient by $G$ by working upwards along a chain of
normal subgroups with factor groups $\Ga$ and $\Gm$ (the
multiplicative group, which will be dealt with in \sref{5sTorus}), and
using \tref{2tSubgroup} in each step. But this only works if~$c$ is
chosen in such a way that $G$ acts on $R_c^\Ga$. This is the case
if~$c$ is a semi-invariant (see after \dref{4dSolvable}). A rather
straightforward strategy for producing a local slice with a
semi-invariant denominator, which would work in the case that the
ground ring $K$ is a field and $R$ is a domain, is the following: One
shows that the denominators of local slices form a $G$-stable
$K$-subspace of $R$. Choose a nonzero finite-dimensional $G$-stable
subspace. Inside this, the fixed space of the unipotent radical is
nonzero, and it decomposes into a direct sum of spaces of
semi-invariants of the torus sitting at the top of $G$. Picking a
semi-invariant in that space yields the desired denominator of a local
slice. Essentially, this is the approach taken by \mycite{Sancho:2000}
for showing that there exists a geometric quotient.

For the following reasons we choose a different, more involved
approach:
\begin{enumerate}
  \renewcommand{\theenumi}{\arabic{enumi}}
\item $K$ may not be a field and $R$ may not be a domain.
\item We wish to obtain a fast and simple algorithm that avoids the
  Gr\"obner basis computations and even the linear algebra that would
  be involved in putting the above strategy into practice.
\end{enumerate}

Instead of assuming $G$ to be a connected solvable group right away,
it is convenient to take $G$ as an affine group scheme with an
embedding of $\Ga$ as a normal subgroup.  Under rather mild
assumptions, this implies that the map $G \to G/\Ga =: H$ splits,
i.e., there is a morphism $\map{\sect}{H}{G}$ of {\em schemes} (not
group schemes) such that $H \xlongrightarrow{\sect} G \to H$ is the
identity, and this yields an isomorphism $G \cong \Ga \times H$ of
schemes (see \mycite[Corollary~1, page~100]{Rosenlicht:1963},
\mycite[Splitting Lemma, page~147]{KMT:1974}). This justifies making
the existence of such a splitting into an {\em assumption}. More
precisely, we work in the following setup:

$G$ and $H$ are affine group schemes over a ring $K$. There is a
morphism $\map{\emb}{\Ga}{G}$ of group schemes (with $\Ga$ the
additive group over $K$) and a morphism $\map{\sect}{H}{G}$ of schemes
such that the composition
$\iso\mbox{:}\ \Ga \times H \xlongrightarrow{(\emb,\sect)} G \times G
\xlongrightarrow{\mult} G$
is an isomorphism of schemes. It is easy to see that we may assume
that $\sect$ takes the identity of $H$ to the identity of $G$. We make
the normal subgroup assumption precise as follows: $H$ acts on $\Ga$
by automorphisms, with the action given by a morphism
$\map{\conj}{H \times \Ga}{\Ga}$, such that the diagram \vspace{-2mm}%
\DIAGV{50}%
{H \times \Ga} \n{} \n{\Earv{\sm{(\conj,\id_H)}}{80}} \n{} \n{\Ga
  \times H} \nn%
{\Sar{\sm{(\sect,\emb)}}} \n{} \n{} \n{} \n{\saR{\sm{(\emb,\sect)}}}
\nn%
{G \times G} \n{} \n{} \n{} \n{G \times G} \nn%
{} \n{\seaR{\sm{\mult}}} \n{} \n{\swaR{\sm{\mult}}} \nn%
{} \n{} \n{G}%
\diag%
commutes. If we write $H = \Spec(A)$ then $\conj$ induces a
homomorphism of $K$-algebras $K[z] \to A[z]$, and it is easy to see
that~$z$ must be sent to $\chi \cdot z$ with $\chi \in A$
invertible. (Viewing $\chi$ as a morphism $H \to \AA^1_K$, it must be
a character of $H$.)

Now let $X = \Spec(R)$ be an affine $K$-scheme with a morphic action
$\map{\act}{G \times X}{X}$. Then the action
$\Ga \times X \xlongrightarrow{(\emb,\id)} G \times X
\xlongrightarrow{\act} X$
induces a homomorphism $\map{\phi}{R}{R[z]}$, and the morphism
$H \times X \xlongrightarrow{(\sect,\id)} G \times X
\xlongrightarrow{\act} X$
(which is not an action) induces $\map{\psi}{R}{A \otimes R}$ (where
this and all other tensor products are over $K$). With $\Ga$ acting
trivially on $H$, it also acts on $H \times X$. The homomorphism
induced by this action is
$\map{\id_A \otimes \phi}{A \otimes R}{(A \otimes R)[z]}$.

As we will see, the following lemma contains everything that is needed
to construct a simple and fast algorithm (\aref{4aSlice}) for
producing a local slice with a semi-invariant denominator, as
discussed above.

\begin{lemma} \label{3lPhi}%
  In the above situation, let $s \in R$ be nonzero and write $\phi(s)
  = \sum_{i=0}^d c_i z^i$ with $c_i \in R$, $c_d \ne 0$.
  \begin{enumerate}
  \item \label{3lPhiA} We have
    \begin{equation} \label{3eqPhiPsi}%
      (\id_A \otimes \phi)\bigl(\psi(s)\bigr) = \sum_{i=0}^d \chi^i
      \psi(c_i) z^i.
    \end{equation}
    Moreover, $\chi^d \psi(c_d) \ne 0$. In particular,
    $\deg\bigl(\psi(s)\bigr) = \deg(s)$, with the degree as defined in
    \sref{1sGa}.
  \item \label{3lPhiB} If~$s$ is a local slice, then so is~$\psi(s)$.
  \end{enumerate}
\end{lemma}

\begin{proof}
  \begin{enumerate}
  \item[\eqref{3lPhiA}] The outer edges of the commutative diagram
    \vspace{-2mm}%
    \DIAGV{42}%
    {H \times \Ga \times X} \n{} \n{} \n{} \n{} \n{} \n{} \n{}
    \n{\Earv{\sm{(\conj,\id_H,\id_X)}}{500}} \n{} \n{} \n{} \n{} \n{}
    \n{} \n{} \n{\Ga \times H \times X} \nn%
    {\Sar{\sm{(\id_H,\emb,\id_X)}}} \n{} \n{} \n{} \n{} \n{} \n{} \n{}
    \n{} \n{} \n{} \n{} \n{} \n{}\n{} \n{}
    \n{\saR{\sm{(\id_\Ga,\sect,\id_X)}}} \nn%
    {H \times G \times X} \n{} \n{} \n{} \n{} \n{} \n{} \n{} \n{} \n{}
    \n{} \n{} \n{} \n{}\n{} \n{} \n{\Ga \times G \times X} \nn%
    {\Sar{\sm{(\id_H,\act)}}} \n{}
    \n{\Esear{\sm{(\sect,\id_G,\id_X)}}} \n{} \n{} \n{} \n{} \n{} \n{}
    \n{} \n{} \n{} \n{} \n{} \n{\Wswar{\sm{(\emb,\id_G,\id_X)}}} \n{}
    \n{\saR{\sm{(\id_\Ga,\act)}}} \nn%
    {H \times X} \n{} \n{} \n{} \n{G \times G \times X} \n{} \n{} \n{}
    \n{} \n{} \n{} \n{} \n{G \times G \times X} \n{} \n{} \n{} \n{\Ga
      \times X} \nn%
    {} \n{} \n{\ \,\,\eseaR{\sm{(\sect,\id)}}} \n{}
    \n{\Sar{\sm{(\id,\act)}}} \n{}
    \n{\Esear{\!\!\!\!^{(\mult,\id_X\!)\!\!\!\!}}}  \n{} \n{} \n{}
    \n{\Wswar{^{(\mult,\id_X\!)\!\!\!\!}}} \n{}
    \n{\saR{\sm{(\id,\act)}}} \n{} \n{\wswaR{\sm{(\emb,\id)}}} \nn%
    {} \n{} \n{} \n{} \n{G \times X} \n{} \n{} \n{} \n{G \times X}\n{}
    \n{} \n{} \n{G \times X} \nn%
    {} \n{} \n{} \n{} \n{} \n{} \n{\eseaR{\sm{\act}}} \n{}
    \n{\Sar{\sm{\act}}} \n{} \n{\wswaR{\sm{\act}}} \nn%
    {} \n{} \n{} \n{} \n{} \n{} \n{} \n{} \n{X}%
    \diag%
    induce the commutative diagram \vspace{-2mm}%
    \DIAGV{53}%
    {} \n{} \n{R} \nn%
    {} \n{\Swar{\phi}} \n{} \n{\Sear{\psi}} \nn%
    {R[z]} \n{} \n{} \n{} \n{A \otimes R} \nn%
    {\Sar{\psi \otimes \id_{K[z]}}} \n{} \n{} \n{} \n{\saR{\id_A
        \otimes \phi}} \nn%
    {(A \otimes R)[z]} \n{} \n{\Earv{z \mapsto \chi z}{70}} \n{} \n{(A
      \otimes R)[z]}%
    \diag%
    From this~\eqref{3eqPhiPsi} follows directly.

    If $\map{\varepsilon}{A}{K}$ is induced by the identity $\Spec(K)
    \to H$ of $H$, then $(\varepsilon \otimes \id_R) \circ \psi =
    \id_R$ and $\varepsilon(\chi) = 1$. So $(\varepsilon \otimes
    \id_R)\bigl(\chi^d \psi(c_d)\bigr) = c_d \ne 0$, and $\chi^d
    \psi(c_d) \ne 0$ follows.
  \item[\eqref{3lPhiB}] By~\eqref{3lPhiA} the highest coefficient of
    $(\id_A \otimes \phi)\bigl(\psi(s)\bigr)$ is $\chi^d \psi(c_d)$.
    So we need to show that if $a \in A \otimes R$ has $\deg(a) < d$,
    then there exists~$k$ such that
    $\chi^{k d} \psi(c_d)^k a \in A \otimes R^\Ga$. Consider the
    commutative diagram \vspace{-2mm}%
    \DIAGV{35}%
    {H \times X} \n{} \n{} \n{}
    \n{\Earv{\sm{(\operatorname{diag},\id_X)}}{140}} \n{} \n{} \n{}
    \n{H \times H \times X} \n{} \n{} \n{}
    \n{\Earv{\sm{(\id_H,\sect,\id_X)}}{160}} \n{} \n{} \n{} \n{H
      \times G \times X} \n{} \n{} \n{}
    \n{\Earv{\sm{(\id_H,\act)}}{140}} \n{} \n{} \n{} \n{H \times X}
    \nn%
    {\Sar{\sm{\pr_2}}} \n{} \n{} \n{} \n{} \n{} \n{} \n{} \n{} \n{}
    \n{} \n{} \n{} \n{} \n{} \n{} \n{} \n{} \n{} \n{} \n{} \n{} \n{}
    \n{} \n{\Sar{\sm{(\sect,\id_X)}}} \nn%
    {X} \n{} \n{} \n{} \n{} \n{} \n{} \n{} \n{\Warv{\sm{\act}}{500}}
    \n{} \n{} \n{} \n{} \n{} \n{} \n{} \n{G \times X} \n{} \n{} \n{}
    \n{\Warv{\sm{(\inv,\id_X)}}{140}} \n{} \n{} \n{} \n{G \times X}
    \nn%
    {\Nar{\sm{\act}}} \n{} \n{} \n{} \n{} \n{} \n{} \n{} \n{} \n{}
    \n{} \n{} \n{} \n{} \n{} \n{} \n{\Nar{\sm{(\mult,\id_X)}}} \nn%
    {G \times X} \n{} \n{} \n{} \n{} \n{} \n{} \n{} \n{} \n{} \n{}
    \n{} \n{} \n{} \n{} \n{} \n{G \times G \times
        X} \n{\Escurve{\sm{(\iso^{-1},\id_X)}}{130}} \nn%
    {\naR{\sm{(\emb,\id_X)}}} \n{} \n{} \n{} \n{} \n{} \n{} \n{} \n{}
    \n{} \n{} \n{} \n{} \n{} \n{} \n{}
    \n{\Nar{\sm{(\emb,\sect,\id_X)}}} \nn%
    {\Ga \times X} \n{} \n{} \n{}
    \n{\Warv{\sm{(\id_{\Ga},\act)}}{140}} \n{} \n{} \n{} \n{\Ga \times
      G \times X} \n{} \n{} \n{}
    \n{\Warv{\sm{(\id_{\Ga},\sect,\id_X)}}{160}} \n{} \n{} \n{} \n{\Ga
      \times H \times X}%
    \diag%
    in which $\map{\inv}{G}{G}$ is the inversion. This induces the
    commutative diagram \vspace{-2mm}%
    \DIAGV{60}%
    {R} \n{} \n{\Earv{\phi}{100}} \n{} \n{R[z]} \n{} \n{\Earv{\psi
        \otimes \id_{K[x]}}{90}} \n{} \n{A[z] \otimes R} \nn%
    {\saR{r \mapsto 1 \otimes r}} \n{} \n{} \n{} \n{} \n{} \n{} \n{}
    \n{\Sar{\eta \otimes \id_R}} \nn%
    {A \otimes R} \n{} \n{\Warv{\mu \otimes \id_R}{75}} \n{} \n{A
      \otimes A \otimes R} \n{} \n{\Warv{\id_A \otimes \psi}{75}} \n{}
    \n{A \otimes R}%
    \diag%
    in which $\map{\mu}{A \otimes A}{A}$ is given by multiplication
    and $\map{\eta}{A[z]}{A}$ is the homomorphism induced by the
    composition
    $H \xlongrightarrow{\sect} G \xlongrightarrow{\inv} G
    \xlongrightarrow{\iso^{-1}} \Ga \times H$.
    First let $a \in R$ with $\deg(a) < d$. By~\eqref{1eqBinom}, all
    coefficients of $\phi(a)$ have degree $< d$, so by~\eqref{3lPhiA}
    the same is true for all coefficients of
    $(\psi \otimes \id_{K[z]})\bigl(\phi(a)\bigr)$. It follows that
    \[
    b := (\eta \otimes \id_R)\Bigl((\psi \otimes
    \id_{K[z]})\bigl(\phi(a)\bigr)\Bigr) \in A \otimes R
    \]
    has degree $< d$. Since~$s$ is a local slice, this implies that
    there is a nonnegative integer~$k$ such that
    $(1 \otimes c_d)^k b \in A \otimes R^\Ga$. The diagram implies
    $(\mu \otimes \id_R)\bigl((\id_A \otimes \psi)(b)\bigr) = 1
    \otimes a$. Moreover,
    \[
    (\mu \otimes \id_R)\bigl((\id_A \otimes \psi)(1 \otimes c_d)\bigr)
    = (\mu \otimes \id_R)\bigl(1 \otimes \psi(c_d)\bigr) = \psi(c_d),
    \]
    and we obtain
    \[
    \psi(c_d)^k (1 \otimes a) = (\mu \otimes \id_R)\bigl((\id_A
    \otimes \psi)\bigl((1 \otimes c_d)^k b)\bigr).
    \]
    By~\eqref{3lPhiA}, applying~$\psi$ to an element of $R^\Ga$ yields
    an element of $A \otimes R^\Ga$, so $(\id_A \otimes \psi)\bigl((1
    \otimes c_d)^k b)\bigr) \in A \otimes A \otimes R^\Ga$, which is
    mapped into $A \otimes R^\Ga$ by $\mu \otimes \id_R$. This shows
    that $\psi(c_d)^k (1 \otimes a) \in A \otimes R^\Ga$, so also
    $\chi^{k d} \psi(c_d)^k (1 \otimes a) \in A \otimes R^\Ga$.

    Now let $\tilde{a} \in A \otimes R$ with $\deg(\tilde{a}) <
    d$. Write~$\tilde{a}$ as a finite sum $\tilde{a} = \sum_i a_i
    \otimes r_i$ with $a_i \in A$ and $r_i \in R$ such that $\deg(r_i)
    < d$. But by the above, there exists~$k$ such that $\chi^{k d}
    \psi(c_d)^k (1 \otimes r_i) \in A \otimes R^\Ga$, from which the
    claim follows.  \qed
  \end{enumerate}
  \renewcommand{\qed}{}
\end{proof}

\section{Unipotent group actions} \label{4sUnipotent}%

In this section we give an algorithm, built on the previous section,
that produces a local slice for an action of an additive group that
appears as a normal subgroup of a connected solvable group, such that
the denominator of the local slice is a semi-invariant. From this, we
construct an algorithm for computing invariants of a unipotent group,
which (for later purposes) is also assumed to be contained in a
connected solvable group.

If $G$ is a connected solvable linear algebraic group over an
algebraically closed field $K$, a lot is known about its structure
(see \mycite[Section~19]{Humphreys}): The factor group $G/U$ by the
unipotent radical is a torus, and $U$ has a chain of subgroups, normal
in $G$, such that all factor groups are isomorphic to $\Ga$. Moreover,
as a variety, $U$ is isomorphic to $\AA^n_K$ ($n = \dim(U)$), with an
isomorphism that is consistent with the subgroup chain just mentioned
(see \mycite[Corollary~2, page~101]{Rosenlicht:1963}). This justifies
making this structure into an assumption for a group scheme in our
more general setting, even though such examples as the group
$\SO_2(K)$ over a field $K$ in which~$-1$ is not a square do not meet
this assumption. In fact, for purposes of stating algorithms, we
assume that the group scheme is given in a way that reflects the above
structure. This is a mild assumption since in practice a connected
solvable group will almost always be given in such a way, for example
if it is defined as a closed subgroup of the group of invertible upper
triangular matrices.

\begin{defi} \label{4dSolvable}%
  A group scheme $G$ over a ring $K$ is said to be in \df{standard
    solvable form} if $G = \Spec\bigl(K[z_1 \upto z_l,t_1 \upto
  t_m,t_1^{-1} \upto t_m^{-1}]\bigr)$ with~$l$ and~$m$ nonnegative
  integers such that:
  \begin{enumerate}
    \renewcommand{\theenumi}{\arabic{enumi}}
  \item \label{4dSolvable1} The closed subscheme $G_i \subseteq G$
    given by the ideal $(z_{i+1} \upto z_l,t_1 - 1 \upto t_m - 1)$ ($i
    = 0 \upto l$) is a normal subgroup. We write $U = G_l$.
  \item \label{4dSolvable2} The morphism $G_i \to \Ga$ given by
    $K[z] \to K[G_i] = K[z_1 \upto z_i], \ z \mapsto z_i$ is a
    morphism of group schemes.
  \item \label{4dSolvable3} With $T := Spec\bigl(K[t_1^{\pm 1} \upto
    t_m^{\pm}]\bigr)$ the $m$-dimensional torus, the morphism $G \to
    T$ given by $t_j \mapsto t_j$ is a morphism of group schemes.
  \end{enumerate}
  If $G$ is in standard solvable form, the torus $T$ acts on each
  $G_i/G_{i-1} \cong \Ga$ by conjugation. The actions are given by
  characters~$\chi_i$ ($i = 1 \upto l$), which are power products of
  the $t_j^{\pm 1}$.
\end{defi}

It is intuitively clear that with
$\Ga \stackrel{\sim}{\longrightarrow} G_1 \to G$ and
$H = \Spec\bigl(K[z_2 \upto z_l,t_1^{\pm 1} \upto t_m^{\pm 1}]\bigr)$
the hypotheses of \lref{3lPhi} are satisfied. The formal verification
of this is a bit tedious and left to the reader.

If a group scheme $G$ in standard solvable form acts on an affine
scheme $\Spec(R)$ with the action given by a homomorphism
$\map{\Phi}{R}{R[z_1 \upto z_l,t_1^{\pm 1}\upto z_m^{\pm 1}]}$, then
an element $c \in R$ is called a \df{semi-invariant} of weight~$\chi$
if $\Phi(c) = \chi \cdot c$, where $\chi = \prod_{j=1}^m t_i^{e_i}$
with~$e_j$ integers. In this case the action extends to $\Spec(R_c)$
by $\Phi(\frac{a}{c^k}) := \chi^{-k} \frac{\Phi(a)}{c^k}$ for $a \in
R$.

We now come to the algorithm for producing a local slice whose
denominator is a semi-invariant. As in \aref{1aSlice} we assume that
it is possible to perform addition, multiplication, and zero testing
of elements of $R$. Notice that the algorithm does not require any
Gr\"obner basis computations and not even linear algebra (unless the
underlying computations in $R$ require Gr\"obner bases).

\begin{alg}[Computation of a local slice with semi-invariant
  denominator] \label{4aSlice} \mbox{}%
  \begin{description}
  \item[\bf Input] A group scheme $G$ in standard solvable form acting
    on an affine scheme $\Spec(R)$ with $R = K[a_1 \upto a_n]$ a
    finitely generated algebra, where the action given is by a
    homomorphism
    $\map{\Phi}{R}{R[z_1 \upto z_l,t_1^{\pm 1}\upto z_m^{\pm
        1}]}$.
    Assume that the characters
    $\chi_i \in K[t_1^{\pm 1}\upto z_m^{\pm 1}]$ as in
    \dref{4dSolvable} are given, and that the subgroup $G_1 \cong \Ga$
    acts nontrivially.
  \item[\bf Output] A local slice $s \in R$ of degree~$d$ with
    denominator~$c$ for the action of $G_1$ such that~$c$ is a
    semi-invariant. Moreover, a homomorphism
    $\map{\pi}{R_c}{R_c^{G_1}}$ of ${R_c^{G_1}}$-algebras, given by
    the $\pi(a_i)$, with $\ker(\pi) = (s)$.
  \end{description}
  \begin{enumerate}
    \renewcommand{\theenumi}{\arabic{enumi}}
  \item \label{4aSlice1} For $i = 1 \upto l$, let
    $\map{\phi_i}{R}{R[z_i]}$ be the homomorphism obtained by
    composing~$\Phi$ with the map fixing~$z_i$, and sending the
    other~$z_j$ to~$0$ and the~$t_j$ to~$1$. Apply \aref{1aSlice} to
    $\phi_1$. Let $s \in R$ be the resulting local slice of degree~$d$
    with denominator~$c$.
  \item \label{4aSlice2} For $i = 2 \upto l$ repeat
    step~\ref{4aSlice3}.
  \item \label{4aSlice3} With~$k$ be the degree of $\phi_i(c)$,
    redefine~$s$ to be the coefficient of~$z_i^k$ in $\phi_i(s)$,
    and~$c$ to be the coefficient of~$z_i^k$ in $\phi_i(c)$. Now~$s$
    is a local slice of degree~$d$ with denominator~$c$, and $c \in
    R^{G_i}$.
  \item \label{4aSlice4} Compute $\Phi(c) \in R[t_1^{\pm 1} \upto
    t_m^{\pm 1}]$ and choose a monomial~$t^*$ occurring in this Laurent
    polynomial. If $d > 1$ and $R$ is not a domain,~$t^*$ has to be
    chosen as the leading monomial of $\Phi(c)$ with respect to an
    arbitrary monomial ordering. Redefine~$s$ to be the coefficient of
    $\chi_1^d \cdot t^*$ in $\Phi(s)$, and~$c$ to be the coefficient
    of~$t^*$ in $\Phi(c)$. Now~$s$ is a local slice of degree~$d$ with
    denominator~$c$, and~$c$ is a semi-invariant with weight~$t^*$.
  \item \label{4aSlice5} With $g := \phi_1(s)$ and
    $f_i := \phi_1(a_i)$, obtain $r_i \in R_c$ by division with
    remainder:
    \[
    f_i = q_i g + r_i
    \]
    with $q_i \in R_c[z_1]$. Then $r_i \in R_c^G$ and $\pi(a_i) = r_i$.
  \end{enumerate}
\end{alg}

\begin{rem} \label{4rSlice}%
  The requirement that $R$ be finitely generated is only used for
  producing a local slice by \aref{1aSlice}. Since local slices always
  exist by \rref{1rSlice}, the algorithm proves the existence of a
  local slice with semi-invariant denominator also when $R$ is not
  finitely generated.
\end{rem}

\begin{proof}[Proof of correctness of \aref{4aSlice}]
  After step~\ref{4aSlice1}, $s$ is a local slice of degree~$d$ with
  denominator $c \in R^{G_1}$. To prove the correctness of
  step~\ref{4aSlice3}, we assume, using induction on $i \ge 2$,
  that~$s$ is a local slice of degree~$d$ with denominator
  $c \in R^{G_{i-1}}$. The factor group $G_i/G_{i-1} \cong \Ga$ acts
  on $R^{G_{i-1}}$ by (the restriction of)~$\phi_i$. So it follows
  by~\eqref{1eqBinom} that the highest coefficient $c'$ of
  $\phi_i(c)$, which is the ``new''~$c$, lies in $R^{G_i}$. We apply
  \lref{3lPhi} to the action of $G_i$ on $\Spec(R)$. The algebra $A$
  from the lemma is $A = K[z_2 \upto z_i]$, and we have to consider
  that map $\map{\psi_i}{R}{R[z_2 \upto z_i]}$ obtained by
  composing~$\Phi$ with the map fixing~$z_2 \upto z_i$ and sending
  $z_1,z_{i+1} \upto z_l$ to~$0$ and all the~$t_j$ to~$1$. The lemma
  tells us that $\psi_i(s)$ is a local slice of degree~$d$ with
  denominator $\psi_i(c)$. Since $c \in R^{G_{i-1}}$, we have
  $\psi_i(c) = \phi_i(c)$. By \lref{3lPhi}, the $z_1^d$-coefficient of
  $\phi_1\bigl(\psi_i(s)\bigr)$ (with~$\phi_1$ applied
  coefficient-wise to $\psi_i(s) \in R[z_2 \upto z_i]$) is
  $\phi_1\bigl(\psi_i(s)\bigr)_d = \psi_i(c)$, so for the
  coefficient~$s'$ of the monomial $z_i^k$ in $\psi_i(s)$ we have
  \[
  \phi_1(s')_d = c'.
  \]
  Taking the coefficient of $z_i^k$ in $\psi_i(s)$ is the same as
  taking the coefficient of $z_i^k$ in $\phi_i(s)$, so~$s'$ is the
  ``new''~$s$. To show that~$s'$ is a local slice, let $a \in R$ with
  $\deg(a) < d$. Since $\phi_i(c)$ is the denominator of the local
  slice $\psi_i(s)$, multiplying~$a$ by a high enough power of
  $\phi_i(c)$ sends it into $R^{G_1}[z]$, so multiplying it by a high
  enough power of~$c'$ sends it into $R^{G_1}$.

  So when the algorithm reaches step~\ref{4aSlice4},~$s$ is a local
  slice with denominator $c \in R^U$ with $U = G_l$. The factor group
  $G/U \cong T$ (the $m$-dimensional torus) acts on $R^U$ with the
  action given by~$\Phi$. So $\Phi(c) \in R[t_1^{\pm 1} \upto t_m^{\pm
    1}]$, and by \lref{4lTorus}, which we prove below, the
  coefficient~$c'$ of any monomial~$t^*$ is a semi-invariant with
  weight~$t^*$. We apply \lref{3lPhi}, so in this case
  $\map{\psi}{R}{R[z_2 \upto z_l,t_1^{\pm 1} \upto t_m^{\pm 1}]}$ is
  the composition of~$\Phi$ with sending~$z_1$ to~$0$. As above, we obtain
  \[
  \phi_1\bigl(\psi(s)\bigr)_d = \chi_1^d \psi(c).
  \]
  This is an equality of (Laurent-)polynomials in $R[z_2 \upto
  z_l,t_1^{\pm 1} \upto t_m^{\pm 1}]$, on which~$\phi_1$ is applied
  coefficient-wise, so comparing the coefficients of $\chi_1^d \cdot
  t^*$ shows that for the coefficient~$s'$ of $\chi_1^d \cdot t^*$ in
  $\psi(s)$ we have
  \[
  \phi_1(s')_d = c'.
  \]
  But~$s'$ is also the coefficient of $\chi_1^d \cdot t^*$ in
  $\Phi(s)$, which is the ``new''~$s$. Since~$c'$ is the ``new''~$c$,
  we are done if we can show that~$s'$ is a local slice of
  degree~$d$. By \lref{3lPhi}, the degree of~$\psi(s)$ is~$d$, so
  $\deg(s') \le d$. But since $c' \ne 0$, the above equation shows
  that $\deg(s') = d$. If $d = 1$ or if $R$ is a domain, then all
  elements of $R$ of degree~$d$ are local slices (see \rref{1rSlice})
  and we are done. If $d > 1$ and $R$ is not a domain (and so~$t^*$ is
  the leading monomial of $\Phi(c)$), then the proof uses
  \lref{3lPhi}\eqref{3lPhiB} and works as above.

  The correctness of step~\ref{4aSlice5} follows directly from
  \tref{1tSlice}\eqref{1tSliceB}.
\end{proof}

The following lemma, which is surely folklore, was used in the above
proof and will be used later, too.

\begin{lemma} \label{4lTorus}%
  Let $T = \Spec\bigl(K[t_1^{\pm 1} \upto t_m^{\pm 1}]\bigr)$ be an
  $m$-dimensional torus over a ring $K$, acting on an affine
  $K$-scheme $X = \Spec(R)$ by a morphism $\map{\act}{T \times
    X}{X}$. With $\map{\Phi}{R}{R[t_1^{\pm 1} \upto t_m^{\pm 1}]}$ the
  induced homomorphism, let $a \in R$ and, for a power product~$t$ of
  the $t_i^{\pm 1}$, let $a_t \in R$ be the coefficient of~$t$ in
  $\Phi(a)$. Then $a_t$ is a semi-invariant of weight~$t$.
\end{lemma}

\begin{proof}
  The commutative diagram \vspace{-3mm}%
  \DIAGV{55}%
  {T \times T \times X} \n{} \n{\Earv{\sm{(\id_T,\act)}}{70}} \n{}
  \n{T \times X} \nn%
  {\Sar{\sm{(\mult,\id_X)}}} \n{} \n{} \n{} \n{\Sar{\sm{\act}}} \nn%
  {T \times X} \n{} \n{\Earv{\sm{\act}}{90}} \n{} \n{X}%
  \diag%
  induces the commutative diagram \vspace{-3mm}%
  \DIAGV{60}%
  {R} \n{} \n{\Earv{\sm{\Phi}}{100}} \n{} \n{R[\underline{t}^{\pm 1}]}
  \nn%
  {\Sar{\sm{\Phi_s}}} \n{} \n{} \n{} \n{\saR{\sm{t_j \mapsto s_j \cdot
        t_j}}} \nn%
  {R[\underline{s}^{\pm 1}]} \n{}
  \n{\Earv{\sm{\id_{K[\underline{s}^{\pm 1}]} \otimes \Phi}}{75}} \n{}
  \n{R[\underline{s}^{\pm 1},\underline{t}^{\pm 1}]}%
  \diag%
  in which $s_1 \upto s_m$ are new indeterminates and~$\Phi_s$ is the
  composition $R \xlongrightarrow{\Phi} R[\underline{t}^{\pm 1}]
  \xlongrightarrow{t_j \mapsto s_j} R[\underline{s}^{\pm 1}]$. So if
  $\Phi(a) = \sum_{e_1 \upto e_m \in \ZZ} a_{e_1 \upto e_m} t_1^{e_1}
  \cdots t_m^{e_m}$, then
  \[
  \sum_{e_1 \upto e_m \in \ZZ} \Phi(a_{e_1 \upto e_m}) s_1^{e_1}
  \cdots s_m^{e_m} = \sum_{e_1 \upto e_m \in \ZZ} a_{e_1 \upto e_m}
  (s_1 t_1)^{e_1}
  \cdots (s_m t_m)^{e_m}.
  \]
  For every $(e_1 \upto e_n) \in \ZZ^m$ the yields
  $\Phi(a_{e_1 \upto e_m}) = a_{e_1 \upto e_m} \cdot t_1^{e_1} \cdots
  t_m^{e_m}$, which was our claim.
\end{proof}

We can now apply \aref{4aSlice} iteratively along a chain of subgroups
and obtain an algorithm for computing $R_c^U$ with $U$ a unipotent
group. The algorithm requires that addition and multiplication of
elements of $R$ are possible, and that for every $c \in R$, zero
testing in $R_c$ is possible. Recall that for a group scheme $G$ in
standard solvable form, $U = G_l$ stands for its unipotent radical,
with the special case $G = U$ possible.

\begin{alg}[Unipotent group invariants] \label{4aUnipotent} \mbox{}%
  \begin{description}
  \item[\bf Input] A group scheme $G$ in standard solvable form acting
    on an affine scheme $\Spec(R)$ with $R = K[a_1 \upto a_n]$ a
    finitely generated algebra, with the action given by a
    homomorphism $\map{\Phi}{R}{R[z_1 \upto z_l,t_1^{\pm 1}\upto
      z_m^{\pm 1}]}$. Assume that the characters $\chi_i \in
    K[t_1^{\pm 1}\upto z_m^{\pm 1}]$ as in \dref{4dSolvable} are
    given.
  \item[Output]
    \begin{itemize}
    \item A semi-invariant $c \in R$, nonzero if $R \ne \{0\}$.
    \item A homomorphism $\map{\pi}{R_c}{R_c^U}$ of $R_c^U$-algebras
      given by the $b_i := \pi(a_i)$, so $R_c^U = K[c^{-1},b_1 \upto
      b_n]$.
    \item Elements $s_1 \upto s_k \in R_c$ such that the map
      $R_c^U[x_1 \upto x_k] \xlongrightarrow{x_i \mapsto s_i} R_c$, is
      an isomorphism and~$\pi$ is equal to the composition $R_c
      \stackrel{\sim}{\longrightarrow} R_c^U[x_1 \upto x_k]
      \xlongrightarrow{x_i \mapsto 0} R_c^U$. In particular,
      $\ker(\pi) = (s_1 \upto s_k)$.
    \end{itemize}
  \end{description}
  \begin{enumerate}
    \renewcommand{\theenumi}{\arabic{enumi}}
  \item \label{4aUnipotent1} Set $c := 1$, $k := 0$, $b_j := a_j$ ($j
    = 1 \upto n$). For $i = 1 \upto l$ repeat
    steps~\ref{4aUnipotent2}--\ref{4aUnipotent4}.
  \item \label{4aUnipotent2} If none of the $\Phi(b_j)$
    involves~$z_i$, skip steps~\ref{4aUnipotent3}--\ref{4aUnipotent4}
    and proceed with the next~$i$.
  \item \label{4aUnipotent3} Apply \aref{4aSlice} to
    $\tilde{G} := \Spec\bigl(K[z_i \upto z_l,t_1^{\pm 1} \upto
    z_m^{\pm 1}]\bigr)$
    and
    $\tilde{R} := K[c^{-1},b_1 \upto b_n] \linebreak \subseteq R_c$,
    with~$\Phi$ extended to $R_c$. Let $\tilde{s}$ be the resulting
    local slice of degree~$d$ with denominator~$\tilde{c}$ and
    $\map{\tilde{\pi}}{\tilde{R}_{\tilde{c}}}{\tilde{R}_{\tilde{c}}^{\tilde{G}}}$
    the resulting homomorphism.
  \item \label{4aUnipotent4} Choose a semi-invariant $c' \in R$ such
    that $(R_c)_{\tilde{c}} = R_{c'}$. This can be done by
    choosing~$c'$ to be a numerator of~$\tilde{c}$ and then
    multiplying it by a high enough power of~$c$ such that it becomes
    a semi-invariant and an $R$-multiple of~$c$. Redefine~$c$ to
    be~$c'$,~$k$ to be $k + 1$,~$b_j$ to be $\tilde{\pi}(b_j)$ ($j = 1
    \upto n$), and set $s_k := \tilde{s}$.
  \end{enumerate}
\end{alg}

\begin{proof}[Proof of correctness of \aref{4aUnipotent}]
  By induction on~$i$ assume that at the beginning of
  step~\ref{4aUnipotent2} the elements $c$, $b_j$, and~$s_j$ are as
  claimed for the output of the algorithm, but with $U$ replaced by
  $G_{i-1}$. Also assume that $\Phi(b_j) \subseteq R_c^{G_{i-1}}[z_i
  \upto z_l,t_1^{\pm 1} \upto t_m^{\pm 1}]$. We will show that, after
  step~\ref{4aUnipotent4}, the same holds with~$i$ replaced by~$i+1$.

  If none of the $\Phi(b_j)$ involves $z_i$, then $R_c^{G_{i-1}} =
  K[c^{-1},b_1 \upto b_n] = R_c^{G_i}$, so step~\ref{4aUnipotent2} is
  correct. After step~\ref{4aUnipotent3}, the map
  $\tilde{R}^{G_i}_{\tilde{c}}[x_{k+1}] \to \tilde{R}_{\tilde{c}}$,
  $x_{k+1} \mapsto \tilde{s}$ is an isomorphism by
  \tref{1tSlice}\eqref{1tSliceA}, and~$\tilde{\pi}$ equals the
  composition $\tilde{R}_{\tilde{c}} \stackrel{\sim}{\longrightarrow}
  \tilde{R}^{G_i}_{\tilde{c}}[x_{k+1}] \xlongrightarrow{x_{k+1}
    \mapsto 0} \tilde{R}^{G_i}_{\tilde{c}}$ by
  \tref{1tSlice}\eqref{1tSliceB}. Since $\tilde{R}_{\tilde{c}} =
  R_{c'}^{G_{i-1}}$, the composition $R_{c'}^{G_i}[x_1 \upto x_{k+1}]
  \xlongrightarrow{x_{k+1} \mapsto s_{k+1}} R_{c'}^{G_{i-1}}[x_1 \upto
  x_k] \xlongrightarrow{x_j \mapsto s_j} R_{c'}$ is an
  isomorphism. The commutative diagram \vspace{-2mm}%
  \DIAGV{60}%
  {R_{c'}} \n{} \n{\!\!\!\!\!\!\!\!\!\!\!\!\!\!\!\!\Earv{\sim}{60}}
  \n{} \n{R_{c'}^{G_{i-1}}[x_1 \upto x_k]} \n{} \n{\Earv{\sim}{22}}
  \n{} \n{R_{c'}^{G_i}[x_1 \upto x_{k+1}]} \nn%
  {} \n{} \n{\eseaR{\sm{\pi}}} \n{} \n{\Sar{\sm{x_j \mapsto 0}}} \n{}
  \n{\ \eseaR{\sm{\id \otimes \tilde{\pi}}}} \n{} \n{\saR{\sm{x_{k+1}
        \mapsto 0}}} \nn%
  {} \n{} \n{} \n{} \n{\,\,\,\,\,R_{c'}^{G_{i-1}}} \n{} \n{} \n{}
  \n{R_{c'}^{G_i}[x_1 \upto x_k]} \nn%
  {} \n{} \n{} \n{} \n{} \n{} \n{\eseaR{\sm{\tilde{\pi}}}} \n{}
  \n{\saR{\sm{x_j \mapsto 0}}} \nn%
  {} \n{} \n{} \n{} \n{} \n{} \n{} \n{} \n{R_{c'}^{G_i}}%
  \diag%
  shows that the ``new'' map~$\pi$ satisfies what is claimed for the
  output of the algorithm. It remains to show that the ``new'' $b_j$
  satisfy $\Phi(b_j) \in R_{c'}^{G_i}[z_{i+1} \upto z_l,t_1^{\pm 1}
  \upto t_m^{\pm 1}]$. The $b_j$ lie in $R_{c'}^{G_i}$, so it suffices
  to prove the statement for any $a \in R_{c'}^{G_i}$. Since
  $R_{c'}^{G_i} \subseteq R_{c'}^{G_{i-1}}$ we have, by induction,
  $\Phi(a) \in R_{c'}^{G_{i-1}}[z_i \upto z_l,t_1^{\pm 1} \upto
  t_m^{\pm 1}]$. We apply \lref{3lPhi} to the action of $G/G_{i-1}$ on
  $R_{c'}^{G_{i-1}}$, with normal subgroup $G_i/G_{i-1} \cong
  \Ga$. The lemma says that applying~$\psi$ to a $\Ga$-invariant
  yields another $\Ga$-invariant, with~$\psi$ formed by
  applying~$\Phi$ followed by sending~$z_i$ to~$0$. So for elements of
  $R_{c'}^{G_i}$, applying~$\psi$ is the same as applying~$\Phi$. So
  indeed $\Phi(a) \in R_{c'}^{G_i}[z_{i+1} \upto z_l,t_1^{\pm 1} \upto
  t_m^{\pm 1}]$.
\end{proof}

\section{Multiplicative group and torus actions} \label{5sTorus}

This section first deals with actions of the multiplicative group. We
define a notion of a local slice and show that its behavior parallels
that of a local slice for an additive group action. Together with the
results from the previous section, this leads to an algorithm for
solvable group actions.

With~$t$ an indeterminate, let $\Gm := \Spec(K[t^{\pm 1}])$ be the
multiplicative group over a ring $K$, acting morphically on an affine
$K$-scheme $\Spec(R)$. The action is given by a homomorphism
$\map{\phi}{R}{R[t^{\pm 1}]}$ of $K$-algebras. If $c \in R$ is a
semi-invariant of weight $\chi = t^k$ (with~$k$ an integer) then the
map
$\mapl{\phi_c}{R_c}{R_c[t^{\pm 1}]}{a/c^e}{\chi^{-e} \phi(a)/c^e}$,
which will also be written as~$\phi$, defines a $\Gm$-action on
$\Spec(R_c)$. For $a \in R$ we write
\[
\deg(a) := \max\{|k| \mid t^k \ \text{occurs in} \ \phi(a)\},
\]
with $\deg(0) := 0$. So the invariant ring $R^\Gm$ consists of the
elements of degree~$0$. We now define the notion of a local slice for
the multiplicative group, which plays a very similar role as a local
slice for the additive group.

\begin{defi} \label{5dSlice}%
  In the above situation, let $0 \ne c \in R$ be a semi-invariant. An
  element $s \in R_c$ is called a \df{local slice} of degree $d > 0$
  with denominator~$c$ if
  \begin{enumerate}
    \renewcommand{\theenumi}{\roman{enumi}}
  \item \label{5dSlice1} $s$ is a semi-invariant of weight $t^{-d}$,
  \item \label{5dSlice2} $s$ is invertible (in $R_c$), and
  \item \label{5dSlice3} every $a \in R_c$ with $\deg(a) < d$ lies in
    $R_c^\Gm$.
  \end{enumerate}
\end{defi}

\begin{prop} \label{5pSlice}%
  In the above situation, if the action is nontrivial, a local slice
  exists.
\end{prop}

\begin{proof}
  If there exists a nonzero nilpotent semi-invariant $c \in R$, then
  $0 = 1 \in R_c$ is a local slice of weight, say, $t$. So we may
  assume that no nonzero semi-invariant is nilpotent. Choose a
  semi-invariant $c \in R$ of positive degree and a semi-invariant
  $s \in R_c$ of {\em minimal} positive degree~$d$. We can choose a
  numerator $s' \in R$ of~$s$ that is a semi-invariant (see
  step~\ref{4aUnipotent4} in \aref{4aUnipotent}). Replacing~$c$ by
  $c s'$ we may assume that~$s$ is invertible, and replacing, if
  necessary,~$s$ by $s^{-1}$, we may assume that~$s$ has weight
  $t^{-d}$. Let $a \in R_c$ be an element of degree $< d$. By
  \lref{4lTorus}, the coefficient of a monomial $t^k$ in $\phi(a)$
  occurring is a semi-invariant of weight $t^k$. Since
  $|k| \le \deg(a) < d$ the minimality of~$d$ implies $k = 0$, so
  $a \in R_c^\Gm$.
\end{proof}


We will present an algorithm for producing a local slice, which
actually does more: It produces a local slice that is a semi-invariant
with respect to a torus containing the multiplicative group. Since we
are still dealing with a single copy of $\Gm$, we postpone presenting
the algorithm and first prove a theorem which shows the usefulness of
local slices and which parallels \tref{1tSlice} and so implies that
the morphism $\Spec(R_c) \to \Spec\bigl((R_c)^\Gm\bigr)$ is an
excellent quotient with fibers $\Spec(K[x^{\pm 1}])$.

\begin{theorem} \label{5tSlice}%
  For a nontrivial action of the multiplicative group $\Gm$ over a
  ring $K$ on an affine $K$-scheme $\Spec(R)$, given by a homomorphism
  $\map{\phi}{R}{R[t^{\pm 1}]}$, let~$s$ be a local slice of
  degree~$d$ with denominator~$c$.
  \begin{enumerate}
  \item \label{5tSliceA} The homomorphism
    $(R_c)^\Gm[y^{\pm 1}] \to R_c$, sending the indeterminate~$y$
    to~$s$, is an isomorphism. We write
    $\map{\psi}{R_c}{(R_c)^\Gm[y^{\pm 1}]}$ for the inverse
    isomorphism.
  \item \label{5tSliceB} The composition
    \[
    \pi\mbox{:}\ R_c \xlongrightarrow{\psi} (R_c)^\Gm[y^{\pm 1}]
    \xlongrightarrow{y \mapsto 1} (R_c)^\Gm
    \]
    is a homomorphism of $(R_c)^\Gm$-algebras with $\ker(\pi) = (s -
    1)$. In particular,~$\pi$ is surjective. For $a \in R_c$, $\pi(a)$
    is given by substituting $t = \sqrt[d]{s}$ in $\phi(a)$, which
    makes sense because $\phi(a) \in R[t^{\pm d}]$.
  \item \label{5tSliceC} The composition
    \[
    R_c \stackrel{\phi}{\longrightarrow} R_c[t^{\pm 1}]
    \stackrel{\pi}{\longrightarrow} (R_c)^\Gm[t^{\pm 1}]
    \]
    (with~$\pi$ applied coefficient-wise) is injective and makes
    $(R_c)^\Gm[t^{\pm 1}]$ into an $R_c$-module that is generated
    by~$d$ elements.
  \item \label{5tSliceD} Let $B$ be a ring with a homomorphism
    $(R_c)^\Gm \to B$. Then
    \[
    (B \otimes_{(R_c)^\Gm} R)^\Gm = B \otimes 1.
    \]
  \end{enumerate}
\end{theorem}

\begin{proof}
  \begin{enumerate}
  \item[\eqref{5tSliceA}] Let $f \in (R_c)^\Gm[y^{\pm 1}]$ be a
    Laurent polynomial with $f(s) = 0$. Then
    \[
    0 = \phi\bigl(f(s)\bigr) = f\bigl(\phi(s)\bigr) = f(s t^{-d}).
    \]
    Since~$s$ is invertible, this implies $f = 0$. This proves
    injectivity. For surjectivity, let $a \in R_c$ by a semi-invariant
    of weight $t^k$ with $k \in \ZZ$. Obtain $k = q d + r$ with $q,r
    \in \ZZ$, $0 \le r < d$ by division with remainder. It follows
    that $s^q a$ is a semi-invariant of degree~$r$ and therefore an
    invariant, so $a = s^q a \cdot s^{-q} \in (R_c)^\Gm[s^{\pm
      1}]$. We also obtain $r = 0$, so~$k$ is divisible by~$d$. For $a
    \in R_c$ arbitrary write $\phi(a) = \sum_k a_k t^k$. By
    \lref{4lTorus}, every $a_k$ is a semi-invariant and therefore lies
    in $(R_c)^\Gm[s^{\pm 1}]$, so the same is true for $a = \sum_k
    a_k$. This shows surjectivity.
  \item[\eqref{5tSliceB}] The first claim is clear. Regarding the
    second claim, we have shown above that $\phi(a) \in R[t^{\pm d}]$
    for $a \in R_c$. For showing that the map that is claimed to be
    equal to~$\pi$ really is~$\pi$, it suffices to check this for~$s$,
    which is straightforward.
  \item[\eqref{5tSliceC}] Let $a \in R_c$. By~\eqref{5tSliceA} we have
    $a = f(s)$ with $f \in (R_c)^\Gm[y^{\pm 1}]$. So
    \[
    \pi\bigl(\phi(a)\bigr) = \pi\bigl(f(\phi(s))\bigr) = \pi\bigl(f(s
    t^{-d})\bigr) = f\bigl(\pi(s) t^{-d}\bigr) = f(t^{-d}),
    \]
    from which~\eqref{5tSliceC} follows.
  \item[\eqref{5tSliceD}] By~\eqref{5tSliceA}, we have $R' := B
    \otimes_{(R_c)^\Gm} R = B[(1 \otimes s)^{\pm 1}]$. By definition,
    $(R')^\Gm = \ker(\phi' - \id)$ with $\map{\phi'}{R'}{R'[t^{\pm
        1}]}$ obtained by tensoring~$\phi$. Let $a \in R'$ and write
    $a = \sum_{i \in \ZZ} b_i (1 \otimes s)^i$ with $b_i \in B$. Then
    \[
    \phi'(a) - a = \sum_i b_i \bigl((1 \otimes t^{-d} s)^i - (1
    \otimes s)^i\bigr) = \sum_i b_i (1 \otimes s)^i (t^{-i d} - 1),
    \]
    which is zero if and only if $b_i = 0$ for $i \ne 0$, i.e., $a \in
    B$. \qed
  \end{enumerate}
  \renewcommand{\qed}{}
\end{proof}

Now we come to the announced algorithm for finding a local slice that
is a semi-invariant with respect to an ambient torus.

\begin{alg}[A local slice for the multiplicative group with a
  semi-invariant denominator] \label{5aSlice} \mbox{}%
  \begin{description}
  \item[Input] A torus $T = \Spec\bigl(K[t_1^{\pm 1} \upto t_m^{\pm
      1}]\bigr)$ over a ring $K$ acting on an affine scheme $\Spec(R)$
    with $R = K[a_1 \upto a_n]$, with the action given by
    $\map{\Phi}{R}{R[t_1^{\pm 1} \upto t_m^{\pm 1}]}$. Assume that the
    subgroup $T_1 \cong \Gm$ given by the ideal $(t_2 - 1 \upto t_m -
    1)$ acts nontrivially.
  \item[Output] A local slice $s \in R_c$ with denominator~$c$ for the
    action of $T_1$, such that~$s$ and~$c$ are semi-invariants of the
    group $T$.
  \end{description}
  \begin{enumerate}
    \renewcommand{\theenumi}{\arabic{enumi}}
  \item \label{5aSlice1} Set $s := 1$, $c := 1$, and $d := 0$.
  \item \label{5aSlice2} While not all $\Phi(a_i)$ (viewed as elements
    of $R_c[t_1^{\pm 1} \upto t_m^{\pm 1}]$) lie in $R_c[t_1^{\pm
      d},t_2^{\pm 1} \upto t_m^{\pm 1}]$, repeat
    steps~\ref{5aSlice3}--\ref{5aSlice5}.
  \item \label{5aSlice3} Choose a monomial $t = t_1^{e_1} \cdots
    t_m^{e_m}$ occurring in a $\Phi(a_i)$ with~$e_1$ not a multiple
    of~$d$, and let $b \in R_c$ be the coefficient of~$t$ in
    $\Phi(a_i)$.
  \item \label{5aSlice4} If $d \ne 0$, use division with remainder to
    obtain $e_1 = q d + r$ with $q,r \in \ZZ$, $0 < |r| \le d/2$. If
    $d = 0$, set $r := e_1$ and $q := 0$. In both cases, set
    $\hat{s} := s^q b$.
  \item \label{5aSlice5} If~$\hat{s}$ is not invertible in $R_c$,
    choose a numerator $s' \in R$ of~$\hat{s}$ that is a
    semi-invariant (see step~\ref{4aUnipotent4} in
    \aref{4aUnipotent}), and redefine~$c$ to be $s' c$. Redefine~$s$
    to be $\hat{s} \in R_c$ if $r < 0$ and $\hat{s}^{-1} \in R_c$ if
    $r > 0$. Finally, set $d := |r|$.
  \end{enumerate}
\end{alg}

\begin{rem} \label{5rSlice}%
  What was said in \rref{4rSlice} also applies to \aref{5aSlice}.
\end{rem}

\begin{proof}[Proof of correctness of \aref{5aSlice}]
  Since $T_1$ acts nontrivially, $d$ becomes positive after the first
  passage through steps~\ref{5aSlice3}--\ref{5aSlice5}, and then
  strictly decreases with each subsequent passage. This guarantees
  termination. By \lref{4lTorus},~$b$ from step~\ref{5aSlice3} is a
  semi-invariant of weight~$t$. It follows that after
  step~\ref{5aSlice5},~$s$ is a semi-invariant with weight
  $t_1^{-d} \cdot$(a power product of $t_2 \upto t_m$). Moreover, in
  step~\ref{5aSlice5}, $s' c$ is nonzero since $b \in R_c$ is nonzero
  and therefore also $\hat{s}$, since~$s$ is invertible. Clearly the
  ``new''~$s$ is invertible. (But although the ``new''~$c$ is nonzero,
  it may happen that $R_c$ becomes the zero-ring.) It remains to show
  that~\eqref{5dSlice3} from \dref{5dSlice} holds when the algorithm
  terminates. An element $a \in R_c$ can be written as
  $a = f(a_1 \upto a_n)/c^k$ with~$f$ a polynomial over $K$. Since
  also~$c$ can be written as a polynomial in the $a_i$, the
  termination condition implies
  $\Phi(a) \in R_c[t_1^{\pm d},t_2^{\pm 1} \upto t_m^{\pm 1}]$. So
  if~$a$ has degree $< d$ (with respect to the $T_1$-action), then
  $a \in R_c^{T_1}$, and the proof is complete.
\end{proof}
  
We can now apply \aref{5aSlice} iteratively to the multiplicative
groups in a torus and combine it with \aref{4aUnipotent}. Thus we
obtain an algorithm for computing $R_c^G$ with $G$ a group scheme in
standard solvable form. As \aref{4aUnipotent}, the algorithm requires
that addition and multiplication of elements of $R$ are possible, and
that for every $c \in R$, zero testing in $R_c$ is possible.

\begin{alg}[Solvable group invariants] \label{5aSolvable} \mbox{}%
  \begin{description}
  \item[\bf Input] A group scheme $G$ in standard solvable form (see
    \dref{4dSolvable}, whose notation we adopt), acting on an affine
    scheme $\Spec(R)$ with $R = K[a_1 \upto a_n]$ a finitely generated
    algebra, with the action given by a homomorphism
    $\map{\Phi}{R}{R[z_1 \upto z_l,t_1^{\pm 1}\upto z_m^{\pm
        1}]}$. Assume that the characters $\chi_i \in K[t_1^{\pm
      1}\upto z_m^{\pm 1}]$ as in \dref{4dSolvable} are given.
  \item[Output]
    \begin{itemize}
    \item A semi-invariant $c \in R$, nonzero if $R \ne \{0\}$.
    \item A homomorphism $\map{\pi}{R_c}{(R_c)^G}$ of
      $(R_c)^G$-algebras given by the $b_i := \pi(a_i)$ and by $b :=
      \pi(c)$. So
      \[
      (R_c)^G = K[b^{-1},b_1 \upto b_n].
      \]
    \item Elements $u_1 \upto u_k,s_1 \upto s_r \in R_c$ such that the
      map
      \[
      (R_c)^G[x_1 \upto x_k,y_1^{\pm 1} \upto y_r^{\pm 1}]
      \xlongrightarrow [y_j \mapsto s_j]{x_i \mapsto u_i} R_c
      \]
      (with~$x_i$, $y_j$ indeterminates) is an isomorphism, and~$\pi$
      is equal to the composition
      $R_c \stackrel{\sim}{\longrightarrow} (R_c)^G[x_1 \upto
      x_k,y_1^{\pm 1} \upto y_r^{\pm 1}] \xlongrightarrow[y_j \mapsto
      1]{x_i \mapsto 0} (R_c)^G$. So
      \[
      \ker(\pi) = (u_1 \upto u_k,s_1 - 1 \upto s_r - 1).
      \]
    \end{itemize}
  \end{description}
  \begin{enumerate}
    \renewcommand{\theenumi}{\arabic{enumi}}
  \item \label{5aSolvable1} Apply \aref{4aUnipotent}. Let $c \in R$ be
    the resulting semi-invariant,~$b_i$ the image of the~$a_i$ under
    the map $\map{\pi}{R_c}{R_c^U}$, and rename the elements $s_1
    \upto s_k$ from \aref{4aUnipotent} to $u_1 \upto u_k$.  Set $b :=
    c$ and $r := 0$. For $j = 1 \upto m$ repeat
    steps~\ref{5aSolvable2}--\ref{5aSolvable4}.
  \item \label{5aSolvable2} If none of the $\Phi(b_i)$
    involves~$t_j$, skip steps~\ref{5aSolvable3}--\ref{5aSolvable4}
    and proceed with the next~$j$.
  \item \label{5aSolvable3} Apply \aref{5aSlice} to $T :=
    \Spec\bigl(K[t_j^{\pm 1} \upto z_m^{\pm 1}]\bigr)$ acting on
    $\tilde{R} := K[b^{-1},b_1 \upto b_n] \subseteq R_c$ by~$\Phi$,
    extended to $R_c$. Let $\tilde{s}$ be the resulting local slice of
    degree~$d$ with denominator~$\tilde{c}$.
  \item \label{5aSolvable4} Choose a semi-invariant $c' \in R$ such
    that $(R_c)_{\tilde{c}} = R_{c'}$ (see step~\ref{4aUnipotent4} in
    \aref{4aUnipotent}). For $i = 1 \upto n$, obtain~$b'_i$ by
    substituting $t_j = \sqrt[d]{\tilde{s}}$ and $t_{j+1} = \cdots =
    t_m = 1$ in $\Phi(b_i)$, which lies in $R_{c'}[t_j^{\pm
      d},t_{j+1}^{\pm 1} \upto t_m^{\pm 1}]$. If $c' = c^e \tilde{c}
    \in R_{c'}$ (such an~$e$ exists by the choice of~$c'$),
    obtain~$b'$ by doing the same with $b^e \tilde{c}$ instead
    of~$b_i$. Set $c := c'$, $b_i := b'_i$, $b := b'$, $s_{r+1} :=
    \tilde{s}$, and $r := r + 1$.
  \end{enumerate}
\end{alg}

\begin{proof}[Proof of correctness of \aref{5aSolvable}]
  By induction on~$j$ assume that at the beginning of
  step~\ref{5aSolvable2} the elements $c$, $b_i$, $b$, $u_i$,
  and~$s_i$ are as claimed for the output of the algorithm, but with
  $G$ replaced by the subgroup $G'_{j-1}$ given by the ideal $(t_j - 1
  \upto t_m - 1)$. In particular, $\tilde{R} = (R_c)^{G'_{j-1}}$. Also
  assume that $b \in R_c$ is invertible. By step~\ref{5aSolvable1},
  this is true for $j = 1$. We will show that after
  step~\ref{5aSolvable4}, the same holds with~$j$ replaced by~$j+1$.

  If none of the $\Phi(b_i)$ involves $t_j$, then $(R_c)^{G'_{j-1}} =
  \tilde{R} = (R_c)^{G'_j}$, so step~\ref{5aSolvable2} is
  correct. In step~\ref{5aSolvable3}, \aref{5aSlice} can be applied
  since the restriction of $\Phi$ to $\tilde{R}$ defines an action of
  $T$ on $\Spec(\tilde{R})$. After step~\ref{5aSolvable3}, the map
  $(\tilde{R}_{\tilde{c}})^{G'_j}[y_{r+1}^{\pm 1}] \to
  \tilde{R}_{\tilde{c}}$, $y_{r+1} \mapsto \tilde{s}$ is an
  isomorphism by \tref{5tSlice}\eqref{5tSliceA}. We have $\tilde{c}
  \in \tilde{R} = (R_c)^{G'_{j-1}}$, so
  \[
  \tilde{R}_{\tilde{c}} = \bigl((R_c)^{G'_{j-1}}\bigr)_{\tilde{c}} =
  \bigl((R_c)_{\tilde{c}}\bigr)^{G'_{j-1}} = (R_{c'})^{G'_{j-1}}
  \]
  with~$c'$ from step~\ref{5aSolvable4}. So the above isomorphism is a
  map $(R_{c'})^{G'_j}[y_{r+1}^{\pm 1}]
  \stackrel{\sim}{\longrightarrow} (R_{c'})^{G'_{j-1}}$. It follows that
  also the composition
  \[
  (R_{c'})^{G'_j}[x_1 \upto x_k,y_1^{\pm 1} \upto y_{r+1}^{\pm 1}]
  \xlongrightarrow{y_{r+1} \mapsto \tilde{s} = s_{r+1}}
  (R_{c'})^{G'_{j-1}}[x_1 \upto x_k,y_1^{\pm 1} \upto y_r^{\pm 1}]
  \xlongrightarrow[y_i \mapsto s_i]{x_i \mapsto u_i} R_{c'}
  \]
  is an isomorphism.

  Moreover, by \tref{5tSlice}\eqref{5tSliceB}, the composition
  \[
  \tilde{\pi}\mbox{:}\ (R_{c'})^{G'_{j-1}}
  \stackrel{\sim}{\longrightarrow} (R_{c'})^{G'_j}[y_{r+1}^{\pm
    1}] \xlongrightarrow{y_{r+1} \mapsto 1} (R_{c'})^{G'_j}
  \]
  is given by applying~$\Phi$ and then substituting $t_j =
  \sqrt[d]{\tilde{s}}$ and $t_{j+1} = \cdots = t_m = 1$. So in
  step~\ref{5aSolvable4}, $b'_i = \tilde{\pi}(b_i) =
  \tilde{\pi}\bigl(\pi(a_i)\bigr)$ and $b' = \tilde{\pi}(b^e
  \tilde{c}) = \tilde{\pi}\bigl(\pi(c)^e \pi(\tilde{c})\bigr) =
  \tilde{\pi}\bigl(\pi(c')\bigr)$, where the second equality holds
  since $\tilde{c} \in (R_{c'})^{G'_{j-1}}$. From this we also see
  that $b'$ is invertible in $R_{c'}$ since~$b$, $\tilde{c}$
  and~$\tilde{s}$ are invertible and they are semi-invariants, so
  applying~$\tilde{\pi}$ means multiplying by a power
  of~$\tilde{s}$. It remains to show that $\tilde{\pi} \circ \pi$ is
  equal to the composition $R_{c'} \stackrel{\sim}{\longrightarrow}
  (R_{c'})^{G'_j}[x_1 \upto x_k,y_1^{\pm 1} \upto y_{r+1}^{\pm 1}]
  \xlongrightarrow[y_i \mapsto 1]{x_i \mapsto 0} (R_{c'})^{G'_j}$.
  But this follows from the commutative diagram \vspace{-2mm}%
  \DIAGV{60}%
  {R_{c'}} \n{\Earv{\sim}{30}} \n{} \n{}
  \n{(R_{c'})^{G'_{j-1}}[\underline{x},y_1^{\pm 1} \upto y_r^{\pm 1}]}
  \n{} \n{} \n{\Earv{\sim}{22}} \n{} \n{}
  \n{(R_{c'})^{G'_j}[\underline{x},y_1^{\pm 1} \upto y_{r+1}^{\pm 1}]}
  \nn%
  {} \n{} \n{\eseaR{\sm{\pi}}} \n{} \n{\Sar{\sm{\begin{array}{r} x_i
          \mapsto 0 \\ y_i \mapsto 1 \end{array}}}} \n{}
  \n{\eseaR{\sm{\id \otimes \tilde{\pi}}}} \n{} \n{}
  \n{\saR{\sm{y_{r+1} \mapsto 1}}} \nn%
  {} \n{} \n{} \n{} \n{\,\,\,\,\,\,\,\,(R_{c'})^{G'_{j-1}}} \n{} \n{}
  \n{} \n{} \n{} \n{(R_{c'})^{G'_j}[\underline{x},y_1^{\pm 1} \upto
    y_r^{\pm 1}]} \nn%
  {} \n{} \n{} \n{} \n{} \n{} \n{\eseaR{\sm{\tilde{\pi}}}} \n{} \n{}
  \n{\saR{\sm{\begin{array}{r} x_i \mapsto 0 \\ y_i \mapsto
          1 \end{array}}}} \nn%
  {} \n{} \n{} \n{} \n{} \n{} \n{} \n{} \n{} \n{(R_{c'})^{G'_j}}%
  \diag%
  This completes the proof.
\end{proof}

\aref{5aSolvable} has been implemented in the computer algebra system
MAGMA~[\citenumber{magma}]. The implementation is limited to the case
that $R$ is a polynomial ring over a field of characteristic~$0$.

We finish this section by giving a summary of the results obtained
about quotients by solvable groups.

\begin{theorem} \label{5tMain}%
  Let $G$ be a group scheme over a ring $K$ in standard solvable form
  acting on a nonempty affine $K$-scheme $\Spec(R)$. (Recall that
  every connected solvable linear algebraic group over an
  algebraically closed field can be brought into standard solvable
  form.)
  \begin{enumerate}
  \item \label{5tMainA} There exists a nonzero semi-invariant
    $c \in R$ such that the morphism
    $X := \Spec(R_c) \to Y := \Spec\bigl((R_c)^G\bigr)$ is an
    excellent quotient by $G$ with fibers
    $F = \linebreak \Spec(K[x_1 \upto x_k,y_1^{\pm 1} \upto y_r^{\pm
      1}])$.
    If $G$ is unipotent, then $r = 0$. In particular, $X \to Y$ is a
    universal geometric quotient, and it is a faithfully flat
    morphism.
  \item \label{5tMainB} Assume $R_c \ne \{0\}$. The
    $(R_c)^G$-homomorphism $\map{\pi}{R_c}{(R_c)^G}$ induced by the
    cross section $Y \to X$ (see
    \rref{2rExcellent}\eqref{2rExcellentC}) has a kernel that is
    generated by~$k + r$ elements. Moreover,
    $\dim(R_c^G) = \dim(R_c) - (k + r)$. So if $K$ is a field and $R$
    or $R_c$ is a complete intersection, then also $(R_c)^G$ is a
    complete intersection.
  \item \label{5tMainC} If $R$ is finitely generated over $K$ and if
    it is possible to carry out the multiplication and addition of
    elements of $R$ and zero testing of elements of $R_{c'}$ for
    $c' \in R$, then \aref{5aSolvable} computes $(R_c)^G$ and a map
    $(R_c)^G[x_1 \upto x_k,y_1^{\pm 1} \upto y_r^{\pm 1}] \to R_c$
    defining the isomorphism
    $X \stackrel{\sim}{\longrightarrow} F \times Y$ that comes with
    the excellent quotient. The algorithm also computes the above
    homomorphism $\map{\pi}{R_c}{(R_c)^G}$ of $(R_c)^G$-algebras and
    its kernel. So $(R_c)^G$ requires at most as many generators as
    $R_c$. \aref{5aSolvable} does not require any Gr\"obner basis
    computations, unless they are necessary for the operations in $R$.
  \item \label{5tMainD} If $R$ is an integral domain, then $\Quot(R)^G
    = \Quot\bigl((R_c)^G\bigr)$. So \aref{5aSolvable} also computes
    the invariant field $K(X)^G = \Quot(R)^G$.
  \end{enumerate}
\end{theorem}

\begin{proof}
  \begin{enumerate}
  \item[\eqref{5tMainA}] The existence of~$c$ follows by induction on
    $l + m$ (the number of factors of type $\Ga$ or $\Gm$ in $G$),
    using Remarks~\ref{4rSlice} and~\ref{5rSlice} for the existence of
    local slices with semi-invariant denominators,
    Theorems~\ref{1tSlice} and~\ref{5tSlice} to show that the quotient
    by the first normal subgroup of type $\Ga$ or $\Gm$ is excellent,
    and \tref{2tSubgroup} to set the induction in motion. The other
    properties of the quotient follow from \tref{2tGeo}.
  \item[\eqref{5tMainB}] By definition,~$\pi$ is the composition of
    the isomorphism
    \[
    R_c \xlongrightarrow{\sim} K[x_1 \upto x_k,t_1^{\pm 1} \upto
    t_r^{\pm 1}] \otimes (R_c)^G = (R_c)^G[x_1 \upto x_k,t_1^{\pm 1}
    \upto t_m^{\pm 1}]
    \]
    with the homomorphism $(R_c)^G[x_1 \upto x_k,t_1^{\pm 1} \upto
    t_r^{\pm 1}] \to (R_c)^G$ sending~$x_i$ to~$0$ and~$t_j$
    to~$1$. The latter map has kernel $(x_1 \upto x_k,t_1 - 1 \upto
    t_r - 1)$. It follows that also the kernel of~$\pi$ is generated
    by~$k + r$ elements. The statement on dimension follows from the
    above isomorphism. If $R$ is a complete intersection, then so is
    $R_c$ since $R_c \cong R[x]/(c x - 1)$ and $\dim(R_c) = \dim(R)$,
    the equality following from the fact that complete intersections
    are equidimensional. Now the statement on complete intersections
    is a consequence of the other statements from~\eqref{5tMainB}.
  \item[\eqref{5tMainC}] See the proof of correctness of
    \aref{5aSolvable}.
  \item[\eqref{5tMainD}] This follows from~\eqref{5tMainA} and
    \lref{2lField}. \qed
  \end{enumerate}
  \renewcommand{\qed}{}
\end{proof}

\exref{2exSO2} shows that the hypothesis that $G$ be in standard
solvable form cannot be dropped from \tref{5tMain}: If $K$ is not an
algebraically closed field, it does not suffice that $G$ is connected
and solvable.

\begin{rem} \label{5rQuasiaffine}%
  The existence of an excellent quotient extends to quasi-affine
  schemes. In fact, let $X$ be a Noetherian scheme and $U \subseteq X$
  an open subscheme. Then $U$ is isomorphic to a (schematically) dense
  open subscheme of $\tilde{X} := \Spec\bigl(\gs{U}\bigr)$ (see
  \mycite[Proposition~13.80]{Goertz:Wedhorn:2010}%
  ) and we may assume $U \subseteq \tilde{X}$. Moreover, let $G$ be a
  group scheme acting on $U$. By the definition of $\tilde{X}$, the
  action extends to it. Observe that $\tilde{X}$ need not be of finite
  type even if $X$ is, but that is not an obstacle to the validity of
  \tref{5tMain}\eqref{5tMainA}. So if $G$ is in standard solvable
  form, there exists a nonzero semi-invariant $c \in \gs{U} =: R$ and
  an excellent quotient $\Spec(R_c) \to Y$. If $U$ is reduced then
  $\Spec(R_c)$ is nonempty, and the same is true for
  $U_c := U \cap \Spec(R_c)$. In any case, $U_c$ is a $G$-stable open
  subscheme of $U$ and of $\tilde{X}$. By
  \rref{2rExcellent}\eqref{2rExcellentG}, $U_c$ admits an excellent
  quotient by $G$ with fibers
  $F = \Spec(K[x_1 \upto x_k,y_1^{\pm 1} \upto y_r^{\pm 1}])$.

  Moreover, when we are working over an algebraically closed field,
  the following is true: If $X$ is an irreducible algebraic variety
  with an action of a connected solvable group $G$, then by
  \mycite[Theorem~2]{Popov:2016} there exists an affine $G$-variety
  $Y$ that is isomorphic (as a $G$-variety) to a dense open subset
  $U \subseteq X$. So as above $X$ has a $G$-stable dense open subset
  that admits an excellent quotient. Thus we recover a result of
  \mycite[Theorem~3]{Popov:2016}.
\end{rem}

\section{A converse} \label{6sConverse}

In this section we ask whether the assertions of \tref{5tMain} are
limited to actions of connected solvable groups. It is fairly clear
(and stated in \tref{6tConverse}\eqref{6tConverseA}) that most parts
of \tref{5tMain} extend to the case in which the $G$-orbits are in
fact orbits of a connected solvable subgroup. Extended in this way,
\tref{5tMain} actually has a converse, which is stated in
\tref{6tConverse}\eqref{6tConverseB}.

Since the proof of the converse requires a result of
\mycite{Borel:1985}, which is only proved over algebraically closed
fields, we assume for the rest of this paper that $K$ is an
algebraically closed field. We also assume that varieties and
algebraic groups are reduced. We need the following terminology to
state our result.

\begin{defi} \label{6dEssential}%
  A morphic action $G \times X \to X$ of a linear algebraic group on
  an affine variety is said to be \df{essentially solvable} if there
  exists a nonzero $d \in K[X]$ such that the open subset $X_d$
  where~$d$ does not vanish is $G$-stable, and for every $x \in X_d$
  the $G$-orbit of~$x$ coincides with the $R(G)$-orbit of~$x$, where
  $R(G)$ is the radical. If we can replace $R(G)$ by the unipotent
  radical $R_u(G)$, then the action is said to be \df{essentially
    unipotent}.
\end{defi}

\begin{theorem} \label{6tConverse}%
  Let $G \times X \to X$ be a morphic action of a linear algebraic
  group on an affine variety.
  \begin{enumerate}
  \item \label{6tConverseA} If the action is essentially solvable,
    then the assertions of \tref{5tMain}\eqref{5tMainA},
    \eqref{5tMainB} and~\eqref{5tMainD} hold, except that the element
    $c \in K[X]$ need not be a semi-invariant, but has the property
    that $X_c$ is $G$-stable. Moreover, we have $(K[X]_c)^G =
    (K[X]_c)^{R(G)}$.
  \item \label{6tConverseB} If there is a nonzero $c \in K[X]$ with
    $X_c$ $G$-stable such that all $G$-orbits in $X_c$ are isomorphic
    (as varieties) to a product $\AA^n_K \times T$ with $T$ a torus,
    then the action is essentially solvable. If $T$ is trivial for all
    orbits, the action is essentially unipotent.
  \end{enumerate}
\end{theorem}

\begin{rem*}
  \begin{enumerate}
  \item Notice that the assertion of \tref{5tMain}\eqref{5tMainA}
    imply the hypothesis of part~\eqref{6tConverseB},
    so~\eqref{6tConverseB} contains the converse
    of~\eqref{6tConverseA}.
  \item If $G$ is connected and $X$ is irreducible, then every $d \in
    K[X]$ such that $X_d$ is $G$-stable has to be a semi-invariant
    (see \mycite[Theorem~3.1]{PV}). So in this case the elements~$c$
    and~$d$ from \dref{6dEssential} and \tref{6tConverse} are
    semi-invariants. \remend
  \end{enumerate}
  \renewcommand{\remend}{}
\end{rem*}

\begin{proof}[Proof of \tref{6tConverse}\eqref{6tConverseA}]
  Applying \tref{5tMain} to the action of $R(G)$ on $X_d$ yields a
  nonzero $R(G)$-semi-invariant $c \in K[X_d]$. For every $x \in X_d$
  and $\sigma \in G$ there exists $\tau \in R(G)$ with $\sigma(x) =
  \tau(x)$, so if $c(x) \ne 0$, then
  \[
  c\bigl(\sigma(x)\bigr) = c\bigl(\tau(x)\bigr) =
  \bigl(\tau^{-1}(c)\bigr)(x) = \chi(\tau)^{-1} c(x) \ne 0,
  \]
  with~$\chi$ the character belonging to~$c$. This shows that
  $(X_d)_c$ is $G$-stable. Choose a numerator of $c \in K[X]_d$,
  multiply it by~$d$ and replace~$c$ by the product. Then $(X_d)_c$ is
  replaced by $X_c$ and $K[X_d]_c$ is replaced by $K[X]_c$. So
  \tref{5tMain}\eqref{5tMainA} tells us that $X_c \to Y :=
  \Spec\bigl((K[X]_c)^{R(G)}\bigr)$ is an excellent quotient by $R(G)$
  with fibers $\Spec(K[x_1 \pto x_k,y_1^{\pm 1} \upto y_r^{\pm
    1}]$. Since the orbits of $G$ and $R(G)$ coincide on $X_c$, we
  have $(K[X]_c)^G = (K[X]_c)^{R(G)}$, and it is easy to check that
  the quotient is also an excellent quotient by $G$. Now also
  \tref{5tMain}\eqref{5tMainD} with $R(G)$ replaced by $G$ follows,
  and so does \tref{5tMain}\eqref{5tMainB}, which only makes
  statements about the ring and the invariant ring.
\end{proof}

Part~\eqref{6tConverseB} of \tref{6tConverse} follows directly from
the following lemma. The main ideas of the proof were shown to me by
Hanspeter Kraft. The last statement is due to \mycite{Borel:1985}.

\begin{lemma} \label{5lKraft}%
  Let $G$ be a linear algebraic group over an algebraically closed
  field $K$, acting transitively on the affine variety $F = \AA^n_K
  \times T$, where $T$ is a torus (regarded as a variety). Then also
  the radical $R(G)$ acts transitively on $F$. If $T$ is trivial, then
  the unipotent radical $R_u(G)$ acts transitively on $F$.
\end{lemma}

\begin{proof}
  A short argument shows that since $F$ is irreducible, the identity
  component $G^0$ acts transitively on $F$, so we may assume $G$ to be
  connected.

  Choose $v \in \AA_K^n$ and
  define $\map{\sigma}{G \times T}{T}$ as the composition
  \[
  G \times T \xlongrightarrow{(g,t) \mapsto (g,v,t)} G \times \AA_K^n
  \times T \xlongrightarrow{\act} \AA_K^n \times T
  \xlongrightarrow{\pr_2} T.
  \]
  This does not depend on the choice of~$v$, since for fixed $g \in G$
  and $t \in T$, mapping $v \in \AA_K^n$ to $\pr_2\bigl(g(v,t)\bigr)$
  yields a morphism $\AA_K^n \to T$. But this must be constant, since
  the induced homomorphism $K[T] = K[x_1,x_1^{-1} \upto x_l,x_l^{-1}]
  \to K[\AA^n_K] = K[y_1 \upto y_n]$ maps invertible elements to
  invertible elements, so it maps the $x_i$ to constants. The
  map~$\sigma$ defines a $G$-action on $T$ since for $g,h \in G$ and
  $t \in T$ we have
  \[
  \sigma(g h,t) = \pr_2\bigl(g(h(v,t))\bigr) =
  \pr_2\bigl(g(v',t')\bigr)
  \]
  with $(v',t') := h(v,t)$, and
  \[
  \sigma\bigl(g,\sigma(h,t)\bigr) = \sigma(g,t') =
  \pr_2\bigl(g(v,t')\bigr),
  \]
  so the independence of the choice of~$v$ implies $\sigma(g h,t) =
  \sigma\bigl(g,\sigma(h,t)\bigr)$. The transitivity of the action on
  $F$ implies the transitivity of the action on $T$.

  Let $\map{\phi}{T}{T}$ be an automorphism. Then for $t = (t_1 \upto
  t_l) \in T$ with $0 \ne t_i \in K$ we have
  \[
  \phi(t) = \Bigl(a_1 \prod_{j=1}^l t_j^{e_{1,j}} \upto a_l
    \prod_{j=1}^l t_j^{e_{l,j}}\Bigr)
  \]
  with $0 \ne a_i \in K$ and $e_{i,j} \in \ZZ$. Fix a~$t$ whose
  components~$t_i$ are $m$th roots of unity. Then the same follows for
  the components of $\phi(t) \phi(1)^{-1}$ (with component-wise
  multiplication and $1 := (1 \upto 1)$), and therefore also for
  $\phi(t) \phi(1)^{-1} t^{-1}$. This leaves only finitely many
  possibilities for $\phi(t) \phi(1)^{-1} t^{-1}$. So since $G$ is
  connected, the morphism $G \to T$, $g \mapsto g(t) g(1)^{-1} t^{-1}$
  is constant, so $g(t) g(1)^{-1} t^{-1} = 1$ for all $g \in G$. This
  holds for all $t \in T$ whose components are $m$th roots of unity
  for some~$m$. But since these~$t$ form a dense subset of $T$, it
  extends to all $t \in T$. So $g(t) = g(1) t$ for $g \in G$, $t \in
  T$. This implies that the morphism $\mapl{\pi}{G}{T}{g}{g(1)}$ is a
  homomorphism of algebraic groups. Since $G$ acts transitively on
  $T$, it is surjective. So by \mycite[Corollary~14.11]{Borel}, also
  the restriction $\pi|_{_{R(G)}}$ is surjective.

  It is straightforward to check that for $H := \ker(\pi) \subseteq G$
  the composition
  \[
  H \times \AA^n_K \xlongrightarrow{(h,v) \mapsto (h,v,1)} H \times
  \AA^n_K \times T \xlongrightarrow{\act} \AA^n_K \times T
  \xlongrightarrow{\pr_1} \AA^n_K
  \]
  defines an action. This action is transitive since for
  $v,v' \in \AA^n_K$ there is $g \in G$ with $g(v,1) = (v',1)$, so
  $g \in H$. Therefore also $H^0$ acts transitively on $\AA^n_K$, and
  from this it follows by \mycite{Borel:1985} that the same is true
  for $R_u(H)$. Now let $(v,t) \in \AA^n_K \times T$. By the above
  there exists $g \in R(G)$ such that $\pi(g) = t$, which implies
  $g^{-1}(t) = 1$ and $g^{-1}(v,t) = (v',1)$ with $v' \in \AA^n_K$. If
  $T$ is trivial we may choose~$g$ to be the identity. There also
  exists $h \in R_u(H)$ such that $h(v',1) = (0,1)$, so
  $(h g^{-1}) (v,t) = (0,1)$. Since
  $R_u(H) \subseteq R_u(G) \subseteq R(G)$, the claim follows.
\end{proof}

\newcommand{\SortNoop}[1]{}

\end{document}